\documentclass[twoside,a4paper,reqno,11pt]{amsart} 
\usepackage{amsfonts, amsbsy, amsmath, amssymb, latexsym, xcolor}
\usepackage{mathrsfs,array}
\usepackage[top=28mm,right=30mm,bottom=28mm,left=30mm]{geometry}
\usepackage{stmaryrd}
\usepackage{bm}

\renewcommand{\a}{\alpha}
\renewcommand{\b}{\beta}

 \renewcommand{\L}{\Lambda}
\renewcommand{\l}{\lambda} \renewcommand{\O}{\Omega}

 \renewcommand{\to}{\rightarrow}
 \newcommand{\s}{\sigma}

\newcommand{\la}{\langle}
\newcommand{\ra}{\rangle}

\newcommand{\leqs}{\leqslant}
\newcommand{\geqs}{\geqslant}

 \newcommand{\vs}{\vspace{3mm}}

\newcommand{\GL}{\operatorname{GL}}

\newcommand{\rad}{\operatorname{Rad}}
\makeatletter
\newcommand{\imod}[1]{\allowbreak\mkern4mu({\operator@font mod}\,\,#1)}
\makeatother

\newtheorem{theorem}{Theorem} 
\newtheorem*{theorem*}{Theorem} 
\newtheorem*{conj*}{Conjecture}

\newtheorem{corol}[theorem]{Corollary}

\newtheorem{thm}{Theorem}[section] 
\newtheorem{prop}[thm]{Proposition} 
\newtheorem{lem}[thm]{Lemma}
\newtheorem{cor}[thm]{Corollary}

\theoremstyle{definition}
\newtheorem{rem}[thm]{Remark}
\newtheorem{remk}{Remark}

\begin{document}

\author{Timothy C. Burness}
 \address{T.C. Burness, School of Mathematics, University of Bristol, Bristol BS8 1UG, UK}
 \email{t.burness@bristol.ac.uk}
 
\author{Spencer Gerhardt}
 \address{S. Gerhardt, Department of Mathematics, University of Southern California,
Los Angeles, CA 90089-2532, USA}
 \email{sgerhard@usc.edu}
 
\author{Robert M. Guralnick}
\address{R.M. Guralnick, Department of Mathematics, University of Southern California,
Los Angeles, CA 90089-2532, USA}
\email{guralnic@usc.edu}
  
\title[Topological generation of exceptional algebraic groups]{Topological generation of exceptional  \\ algebraic groups}

\begin{abstract}
Let $G$ be a simple algebraic group over an algebraically closed field $k$ and let 
$C_1, \ldots, C_t$ be non-central conjugacy classes in $G$. In this paper, we consider the problem of determining whether there exist $g_i \in C_i$ such that $\langle g_1, \ldots, g_t \rangle$ is Zariski dense in $G$. First we establish a general result, which shows that if $\O$ is an irreducible subvariety of $G^t$, then the set of tuples in $\Omega$ generating a dense subgroup of $G$ is either empty or dense in $\Omega$. In the special case $\O = C_1 \times \cdots \times C_t$, by considering the   dimensions of fixed point spaces, we prove that this set is dense when $G$ is an exceptional algebraic group and $t \geqs 5$, assuming $k$ is not algebraic over a finite field. In fact, for $G=G_2$ we only need $t \geqs 4$ and both of these bounds are best possible. As an application, we show that many faithful representations of exceptional algebraic groups are generically free. We also establish new results on the topological generation of exceptional groups in the special case $t=2$, which have applications to random generation of finite exceptional groups of Lie type. In particular, we prove a conjecture of Liebeck and Shalev on the random $(r,s)$-generation of exceptional groups.
\end{abstract}

\date{\today}

\thanks{We thank an anonymous referee for their careful reading of the paper and for many helpful comments and suggestions. The third author was partially supported by the NSF grant DMS 1901595.} 

\maketitle

\section{Introduction}\label{s:intro}

In this paper, we study the problem of topologically generating a simple algebraic group $G$ defined over an algebraically closed field, with respect to the Zariski topology on $G$. Recall that a subset of $G$ is a topological generating set if it generates a dense subgroup. We are primarily interested in finding small subsets with this property, where the generators are contained in specified conjugacy classes. We will typically work with the simply connected form of the group, but the isogeny class makes no difference. Indeed, the center of $G$ is contained in the Frattini subgroup, so a subgroup $H$ is dense in $G$ if and only if $HZ/Z$ is dense in $G/Z$, where $Z$ is any central subgroup of $G$. One can also extend the main results presented below to semisimple groups in an obvious way.    

Let $G$ be a simply connected simple algebraic group over an algebraically closed field $k$ of characteristic $p \geqs 0$. The following theorem of Guralnick \cite{gurnato} is presumably well known to the experts.

\begin{theorem*}[\cite{gurnato}]\label{NATO}   
If $p=0$ then 
\[
\{(x,y) \in G \times G \,:\, \overline{\langle x, y \rangle} = G \}
\] 
is a nonempty open subset of $G \times G$.
\end{theorem*}

This result also holds for semisimple groups, and the analogous statement for $t$-tuples with $t \geqs 3$ is an immediate corollary. Note that the conclusion is false in positive characteristic. Indeed, if $k_0 \subseteq k$ is the algebraic closure of the prime field, then the set of $k_0$-points in $G \times G$ is dense and of course any pair of elements in $G(k_0)$  generates a finite subgroup, where
$G(k_0)$ is the set of $k_0$-points in $G(k) = G$. The following extension to positive characteristic is proved in \cite[Theorem 11.7]{GT} (also see \cite[Proposition 4.4] {BGGT} for a generalization to pairs of noncommuting words). 

\begin{theorem*}[\cite{GT}]\label{t:GT} 
If $p>0$ then the set of elements $(x,y) \in G \times G$ such that either  
$\overline{\langle x, y \rangle} = G$, or $\langle x, y \rangle$ contains a conjugate subgroup of the form $G(p^a)$ (possibly twisted), is a nonempty open subset of $G \times G$. 
\end{theorem*}

The fact that $G$ is $2$-generated topologically (as long as $k$ is not algebraic over a finite field) follows from Tits' result \cite{Tits}  that any semisimple algebraic group contains a Zariski dense free subgroup on two generators (a variant of the \emph{Tits alternative}).

Let $V$ be an irreducible variety defined over $k$. We say that a subset  $S$ of $V$ is   \emph{generic} if it contains the complement of a countable union of proper closed subvarieties. Note that if $k'$ is an uncountable algebraically closed field containing $k$ then $S(k')$ is dense in $V(k')$; in particular, if $k$ is itself uncountable then $S(k)$ is dense in $V(k)$. On the other hand, if $k$ is countable, then $S(k)$ may be empty (for instance, consider the countably many 
one-point subvarieties of $V(k)$). If we avoid the  countably many subvarieties of pairs $(x,y) \in G \times G$ such that $|\overline{\langle x, y \rangle}| \leqs n$
for $n=1, 2, \ldots$, then the previous result implies that   
 \[
 \{(x,y) \in G \times G \,  : \, \overline{\langle x, y \rangle} = G \}
 \] 
 is a generic subset of $G \times G$.   See \cite{BGGT} for a stronger result, which establishes the genericity of the set of pairs $(x,y)$ such that $\langle x, y \rangle$ is a \emph{strongly dense} free subgroup of $G$ (that is, $\langle x, y \rangle$ is free and every nonabelian subgroup is dense).    

In this paper, we will focus on the topological generation of simple algebraic groups with respect to generators from a finite number of fixed conjugacy classes, with the aim of extending some of the results highlighted above. To do this, we first consider a somewhat more general situation. 

Let $\Omega$ be an irreducible subvariety of $G^t = G \times \cdots \times G$ ($t \geqs 2$ factors). Here we do not insist that $\O$ is closed, but only open in its closure (that is, we assume $\O$ is locally closed in $G^t$). Some important special cases of interest include $G^t$ itself, $\{g\} \times G$ for $g \in G$ 
non-central and $C_1 \times \cdots \times C_t$ where each $C_i$ is a non-central conjugacy class of $G$. For $x = (x_1, \ldots, x_t) \in G^t$, let $G(x)$ be the closure of the subgroup of $G$ generated by the $x_i$ and define
\begin{align}\label{e:deltap}
\Delta & = \{ x \in \Omega \,:\, G(x) = G\} \nonumber \\
\Delta^{+} & = \{ x \in \Omega \,:\, \dim G(x) >0 \} \\
\L & = \{ x \in \Omega \,:\, \mbox{$G(x) \not\leqs H$ for all $H \in \mathcal{M}$}\} \nonumber
\end{align}
where $\mathcal{M}$ is the set of positive dimensional maximal closed subgroups of $G$. Note that 
\[
\Delta = \Delta^+ \cap \L.
\]

In characteristic zero, the property of topologically generating $G$ is open (see \cite{gurnato}  -- we will also discuss this in Section \ref{s:top}), so $\Delta$ is an open subset of $\Omega$ and therefore empty or dense. As noted above for $\O = G^t$, in positive characteristic this need not be the case. In this setting, let us also observe that if $\O = C_1 \times \cdots \times C_t$ then $\Delta$ is open if at least one of the conjugacy classes contains elements of infinite order (this is a trivial corollary of the previous theorem;  see \cite{GT}). It is also clear from the discussion above that $\Delta$ and $\Delta^+$ are either empty or generic. We refer the reader to \cite[Proposition 2.1]{BGR} for a more general result on the genericity of certain varieties. In this paper, we will focus on groups defined over countable algebraically closed fields in positive characteristic.  

Our first main result addresses the density of $\Delta^+$. Note that if $k'$ is a field extension of $k$, then we can consider $G(k')$, $\O(k')$, $\Delta(k')$, etc., which are defined in the obvious way.

\begin{theorem} \label{t:density}  
Let $G$ be a simply connected simple algebraic group over an algebraically closed field $k$ of characteristic $p \geqs 0$. Assume that $k$ is not algebraic over a finite field. Let $\O$ be an irreducible subvariety of $G^t$.  Then either
\begin{itemize}\addtolength{\itemsep}{0.2\baselineskip}
\item[{\rm (i)}] $\Delta^+$ is a dense subset of $\Omega$; or
\item[{\rm (ii)}] $\Delta^+(k')$ is empty for every field extension $k'$ of $k$.
\end{itemize}
\end{theorem}

The next result shows that the existence of a single $x \in \O$ with $G(x)=G$ implies that the set of such $x$ is dense in $\O$.   

\begin{theorem} \label{t:dense2}  
Let $G$ be a simply connected simple algebraic group over an algebraically closed field $k$ of characteristic $p \geqs 0$. Assume that $k$ is not algebraic over a finite field. Let $\O$ be an irreducible subvariety of $G^t$. Then the following are equivalent:
\begin{itemize}\addtolength{\itemsep}{0.2\baselineskip}
\item[\rm (i)]  $G(x)=G(k')$ for some $x \in \Omega(k')$ and field extension $k'$ of $k$.
\item[{\rm (ii)}]  $\Delta$ is a dense subset of $\O$.
\end{itemize}
Moreover, if either (i) or (ii) hold, then $\L$ contains a nonempty open subset of $\O$. 
\end{theorem} 

\begin{remk}
As noted above, $\Delta$ is open if $p=0$, or if $\O = C_1 \times \cdots \times C_t$ and one of the  $C_i$ consists of elements of infinite order. On the other hand, if $p>0$ and $\Omega$ is defined over $k_0$, the algebraic closure of the prime field,
then $G(x)$ is finite for all $x$ in the dense subset $\O(k_0)$ and thus $\Delta$ (and also $\Delta^+$) is only open in $\O$ if it is empty. By Theorem \ref{t:dense2}, if $\Delta$ is generic, then it is also dense as long as $k$ is not algebraic over a finite field.  Similarly, Theorem \ref{t:density} implies that the same conclusion holds for $\Delta^+$.
\end{remk}

It is well known that if $k$ is not algebraic over a finite field then for every non-central element $g \in G$, there exists $h \in G$ such that $G = \overline{\la g,h \ra}$ (see \cite{gurnato} or \cite{BGGT}). Therefore, as an immediate corollary of Theorem \ref{t:dense2} we get the following.

\begin{corol}\label{c:density}  
Let $G$ be a simply connected simple algebraic group over an algebraically closed field $k$ of characteristic $p \geqs 0$. Assume that $k$ is not algebraic over a finite field.
\begin{itemize}\addtolength{\itemsep}{0.2\baselineskip}
\item[{\rm (i)}] If $t \geqs 2$, then $\{ x \in G^t \, : \, G(x)=G \}$ is a dense subset of $G^t$.
\item[{\rm (ii)}] If $g \in G$ is non-central, then $\{h \in G \,:\, G = \overline{\la g,h \ra}\}$ is a dense subset of $G$.
\end{itemize}
\end{corol} 

By the above, we can see that if $p=0$ then the sets in parts (i) and (ii) of Corollary \ref{c:density} are nonempty and open. The same conclusion holds in (ii) if $p>0$ and $g$ has infinite order.

We now turn our attention to the special case where 
$\Omega = C_1 \times \cdots \times C_t$ and each $C_i$ is a 
non-central conjugacy class.  First, using the main theorem of \cite{GMT}
and the proof of Theorem \ref{t:density}, we obtain the following result.

\begin{corol}  \label{c:gmt}
Let $G$ be a simply connected simple algebraic group over an algebraically closed field $k$ of characteristic $p \geqs 0$. Let $\O = C_1 \times \cdots \times C_t$ with $t \geqs 2$, where the $C_i$ are non-central conjugacy classes of $G$. Then $\Delta^+$ is empty if and only if either
\begin{itemize}\addtolength{\itemsep}{0.2\baselineskip}
\item[{\rm (i)}] $k$ is algebraic over a finite field; or 
\item[{\rm (ii)}] $p > 0$, $t=2$, $C_1C_2$ is a finite union of conjugacy classes of $G$ and 
the pair $C_1,C_2$ is described in \cite[Theorem 1.1]{GMT}. 
\end{itemize}
\end{corol}

There are only a small number of pairs arising in \cite[Theorem 1.1]{GMT} and they exist if and only if the root system of $G$ has both long and short roots.   

\vs

In \cite{Ge0, Ge1}, Gerhardt studies the topological generation of the classical algebraic groups ${\rm SL}_n(k)$ and ${\rm Sp}_{2n}(k)$ by elements in specified conjugacy classes. The conditions for ${\rm SL}_n(k)$ with $n \geqs 3$ are especially nice and just depend upon the action of the conjugacy class representatives on the natural $n$-dimensional module (see \cite[Theorem 1.1]{Ge1}). 

In this paper, we focus on the simple algebraic groups of exceptional type. We will establish our main results by studying the primitive actions of these groups and the dimensions of the corresponding fixed point spaces. Here the key tool is Theorem \ref{t:main3} below (also see \cite[Lemma 2.4]{Ge0}), which relies on the fact that a simple algebraic group has only finitely many conjugacy classes of maximal closed subgroups of positive dimension (see \cite[Corollary 3]{LS04}).

Let $M$ be a maximal closed subgroup of $G$ and consider the natural transitive action of $G$ on the coset variety $X = G/M$. For $g \in G$, let 
\[
X(g) = \{ x \in X \,:\, x^g = x\}
\]
be the fixed point space of $g$ on $X$, which is a subvariety of $X$, and set
\[
\a(G,M,g) = \frac{\dim X(g)}{\dim X}.
\]
This is a natural analogue for algebraic groups of the classical notion of fixed point ratio for actions of finite groups and there is a close connection between this quantity and fixed point ratios for the corresponding finite groups of Lie type. See \cite{Bur1, LLS} for more details. 

\begin{theorem}\label{t:main3} 
Let $G$ be a simple algebraic group over an algebraically closed field and let $M_1, \ldots, M_s$  represent the conjugacy classes of maximal closed subgroups of $G$ of positive dimension. If $C_1, \ldots, C_t$ are non-central conjugacy classes such that
\[
\sum_{i=1}^t \a(G,M_j,g_i)  <  t-1 
\]
for all $j$, where $C_i = g_i^G$, then $\L$ contains a nonempty open subset of $\O = C_1 \times \cdots \times C_t$.
\end{theorem} 

In order to apply this theorem in the context of exceptional algebraic groups, we present the following result on the dimensions of fixed point spaces, which may be of independent interest. We refer the reader to Theorem \ref{t:mainex} for a more detailed statement (the stronger form will be needed for the proof of Theorem \ref{t:main5s} below).

\begin{theorem}\label{t:main4}
Let $G$ be a simple exceptional algebraic group over an algebraically closed field and let $\mathcal{M}$ be the set of positive dimensional maximal closed subgroups of $G$. Then 
\[
\max\{\a(G,M,g) \,:\, \mbox{$g \in G$ non-central, $M \in \mathcal{M}$}\} = \kappa(G),
\]
where 
\[
\begin{array}{cccccc} \hline
G & E_8 & E_7 & E_6 & F_4 & G_2 \\ 
\kappa(G) & 15/19 & 7/9 & 10/13 & 3/4 & 2/3 \\ \hline
\end{array}
\]
\end{theorem}

By combining this with Theorems \ref{t:density}, \ref{t:dense2} and \ref{t:main3}, together with Corollary \ref{c:gmt}, we immediately obtain the following result.  

\begin{theorem}\label{t:main5} 
Let $G$ be a simple exceptional algebraic group over an algebraically closed field that is not algebraic over a finite field and set $\O = C_1 \times \cdots \times C_t$, where each $C_i$ is a non-central conjugacy class of $G$. If $t \geqs 5$ then $\Delta$ is a dense subset of $\O$. Moreover, if $G = G_2$ then the same conclusion holds for $t \geqs 4$.
\end{theorem}

\begin{remk}
We will show below in Theorem \ref{t:best} that the bounds on $t$ in Theorem \ref{t:main5}  are best possible. For example, $\Delta$ is empty when $G = E_8$, $t=4$ and each $C_i$ is the class of long root elements.  
\end{remk}

If we exclude certain classes, then the proof of Theorem \ref{t:main5} yields the following result on triples (with some additional effort, we could remove the prime order assumption). 

\begin{theorem}\label{t:main5s}  
Let $G$ be a simple exceptional algebraic group over an algebraically closed field 
of characteristic $p \geqs 0$ that is not algebraic over a finite field. Let $C_1,C_2,C_3$  be non-central conjugacy classes of $G$ consisting of either unipotent elements or elements of prime order modulo the center of $G$. Then one of the following holds:
\begin{itemize}\addtolength{\itemsep}{0.2\baselineskip}
\item[{\rm (i)}] There exist $x_i \in C_i$ such that $G = \overline{\la x_1, x_2, x_3\ra}$ (and therefore the set of such triples is dense in $C_1 \times C_2 \times C_3$); 
\item[{\rm (ii)}]  Some $C_i$ consists of long root elements (or short root elements if $(G,p)$ is $(F_4,2)$ or $(G_2,3)$); 
\item[{\rm (iii)}] $G=F_4$, $p \ne 2$ and some $C_i$ consists of involutions with centralizer $B_4$.
\end{itemize} 
\end{theorem} 

\begin{remk}\label{r:lie}  
If $\mathfrak{g}$ is the Lie algebra of the simply connected simple algebraic group $G$, then it is natural to consider the existence of generators for $\mathfrak{g}$ in given orbits under the adjoint representation of $G$. This problem was studied in \cite{CW} and \cite{GG}. The answer is essentially the same as in Theorem \ref{t:main5}, except for the cases when the underlying characteristic $p$ is \emph{special} for $G$, which occurs when 
$(G,p)$ is one of $(A_1,2)$, $(B_n,2)$, $(C_n,2)$, $(F_4,2)$ or $(G_2,3)$. 
Note that these are precisely the cases where $\mathfrak{g}$ contains nontrivial non-central proper ideals. See \cite{GG3} for a discussion of the special cases. 
\end{remk}

\begin{remk}\label{r:gs}   
In \cite{GS}, Guralnick and Saxl study the analogous problem for finite simple groups,   determining an upper bound on the number of elements in a fixed conjugacy class that are needed to generate the group. For classical groups, the given upper bound is almost always the dimension of the natural module (and this is best possible, with known exceptions). For the exceptional groups, an upper bound of the form of $\ell + 3$ is given, where $\ell$ is the rank of the ambient algebraic group (for $F_4(q)$, the bound is $\ell + 4= 8$). We conjecture that the bounds in Theorem \ref{t:main5} extend in the obvious way to the corresponding finite simple groups of Lie type (so $5$ conjugates should be sufficient for $E_8(q)$, etc.). Moreover, such bounds would be best possible (see Remark \ref{r:finitebest}).

The bounds presented in \cite{GS} have proved to be useful in a wide range of problems. Indeed, several applications are already given in \cite{GS}. This includes a classification of the irreducible modules for quasisimple groups containing bireflections (and more generally, elements acting with a large fixed point space), as well as a description  of the simple primitive permutation groups of special degrees. The bounds have also played an important role in several papers concerning fixed point ratios and base sizes for almost simple primitive permutation groups (see \cite{fpr4, Base}, for example).
\end{remk} 

Recall that if an algebraic group $G$ acts on a variety $X$, then we say that a closed subgroup $H$ of $G$ is the \emph{generic stabilizer} for this action if there exists a nonempty open subset $X_0$ of $X$ such that the stabilizer $G_x$ is conjugate to $H$ for all $x \in X_0$. In particular, we say that the generic stabilizer is trivial if $H =1$ (more generally, one can consider the stabilizer as a group scheme). By arguing as in \cite{GG}, using the bounds above in Theorem \ref{t:main5}, we can recover a version of \cite[Theorem 2]{GL} for exceptional groups.

In order to state this result, let $V$ be a $kG$-module and set
\begin{equation}\label{e:VG}
V^G = \{v \in V \,:\, \mbox{$v^g = v$ for all $g \in G$}\}.
\end{equation}
We say that $V$ is \emph{generically free} if the generic stabilizer for the action of $G$ on $V$ is trivial.

\begin{theorem} \label{t:generic}  
Let $G$ be a simple exceptional algebraic group over an algebraically closed field $k$ and consider the action of $G$ on a faithful rational finite dimensional $kG$-module $V$. If $\dim V/V^G > d(G)$, where $d(G) = 3(\dim G - {\rm rank}\, G)$, then $V$ is generically free. 
\end{theorem}

By inspection of the irreducible modules of low dimension, it follows by \cite{GL} that the same conclusion holds whenever $V$ is irreducible and $\dim V > \dim G$.  Indeed, with a bit more effort, it is not too hard to use the above result to show that  $d(G)$ can be replaced by $\dim G$ (this is for an exceptional group $G$; the same conclusion almost always holds when $G$ is classical, but there are exceptions -- see \cite{GL}).  We refer the reader to \cite{Ge1} for similar bounds in the case $G = {\rm SL}_{n}(k)$. Much more generally, different methods were used in \cite{GL} to compute generic stabilizers for the action of any simple algebraic group on any finite dimensional irreducible module (in particular, \cite[Theorem 1]{GL} establishes the existence of a generic stabilizer in this situation). 

\vs

We will also use our methods to establish new results on the random generation of finite simple exceptional groups of Lie type. More generally, let $H$ be a finite group and let $I_m(H)$ be the set of elements of order $m$ in $H$. Then for positive integers $r$ and $s$, let
\[ 
\mathbb{P}_{r,s}(H) = \frac{|\{(x,y) \,:\, x \in I_r(H), y \in I_s(H), H = \la x, y \ra\}|}{|I_r(H)||I_s(H)|}
\]
be the probability that $H$ is generated by randomly chosen elements of orders $r$ and $s$ (if $I_r(H)$ or $I_s(H)$ is empty, then we set $\mathbb{P}_{r,s}(H) =0$). It is well known that  every finite simple group is $2$-generated and so there is a particular interest in studying this probability when $H$ is a simple group. It is also natural to assume that both $r$ and $s$ are primes.

Clearly, in this setting, there are no such pairs if $(r,s)=(2,2)$. The case $(r,s)=(2,3)$ has attracted significant attention because the groups with such a generating pair coincide with the images of the modular group ${\rm PSL}_2(\mathbb{Z})$. Here one of the main results is \cite[Theorem 1.4]{LSh96}, which states that if $H$ is a finite simple classical group then
\[
\mathbb{P}_{2,3}(H) \to \left\{\begin{array}{ll}
1 & \mbox{if $H \ne {\rm PSp}_{4}(q)$} \\
1/2 & \mbox{if $H = {\rm PSp}_{4}(p^f)$ and $p \geqs 5$} \\
0 & \mbox{if $H = {\rm PSp}_{4}(p^f)$ and $p \in \{2,3\}$}
\end{array}\right.
\]
as $|H|$ tends to infinity. The analogous result for an exceptional group $H$ is \cite[Theorem 9]{GLLS}, which shows that $\mathbb{P}_{2,3}(H) \to 1$ as $|H| \to \infty$ (with the obvious exception of the Suzuki groups, which do not contain elements of order $3$). The proof of the latter result is based on a more general observation in \cite{GLLS}, which implies that it is sufficient to work in the ambient algebraic group (and check which conjugacy classes are invariant under the corresponding Steinberg endomorphism). For arbitrary primes $r$ and $s$ (with $(r,s) \ne (2,2)$), the main theorem of \cite{LSh} shows that if $H$ is a finite simple group of Lie type, then $\mathbb{P}_{r,s}(H) \to 1$ as the rank of $H$ tends to infinity. Similarly, \cite[Theorem 1.3]{Ge1} states that if $H = {\rm PSL}_{n}(q)$ then $\mathbb{P}_{r,s}(H) \to 1$ as $q$ tends to infinity and 
$r, s$ both divide $|{\rm PSL}_n(q)|$.    

Our main result in this direction is Theorem \ref{t:main7} below, which establishes the random $(r,s)$-generation of finite simple exceptional groups for all appropriate primes $r$ and $s$. The proof relies on an extension of  Theorem \ref{t:main4} on the dimensions of fixed point spaces for primitive actions of exceptional algebraic groups. More precisely, we establish a stronger upper bound on $\a(G,M,g)$ for representatives $g \in G$ of certain ``large" conjugacy classes of $G$. This is the content of Theorem \ref{t:main6}. In order to give a precise statement, we need to introduce some notation.

Let $G$ be a simply connected simple algebraic group over the algebraic closure of a finite field of characteristic $p$ and assume that the type of $G$ is one of the following:
\begin{equation}\label{e:gro}
E_8, \, E_7, \, E_6, \, F_4, \, G_2, \, D_4, \, B_2 \, (p=2).
\end{equation}
Let $\s$ be a Steinberg endomorphism of $G$ such that $G_{\s} = G(q)$ is a finite quasisimple exceptional group of Lie type over $\mathbb{F}_q$, where $q=p^f$ (here $G(q)$ could be twisted). Let $r$ be a prime and define
\begin{align*}
G_{[r]}  & = \{ g \in G \, : \, g^r \in Z(G) \} \\
\mathcal{C}(G,r,q) & = \max\{\dim g^G \,:\, \mbox{$g \in G(q)$ has order $r$ modulo $Z(G)$}\},
\end{align*}
so $\mathcal{C}(G,r,q) \leqs \dim G_{[r]}$. Note that if $\bar{G} = G/Z(G)$ denotes the corresponding adjoint group, then $\dim G_{[r]} = \dim \bar{G}_{[r]}$ and so we can read off the dimension of $G_{[r]}$ from \cite{Law05}. In particular, if $r \geqs h$, where $h$ is the Coxeter number of $G$ (recall that $h+1  =\dim G /{\rm rank}\, G$), then $G$ contains a regular element of order $r$ and so we get 
\[
\dim G_{[r]} = \dim G - {\rm rank}\, G.
\] 
The dimension of $G_{[r]}$ for $r<h$ is recorded in Table \ref{tab:gr} in Section \ref{s:random} and the values for $r \in \{2,3\}$ are as follows:
\[
\begin{array}{cccccccc} \hline
 & E_8 & E_7 & E_6 & F_4 & G_2 & D_4 & B_2 \\ \hline
2 & 128 & 70 & 40 & 28 & 8 & 16 & 6 \\
3 & 168 & 90 & 54 & 36 & 10 & 18 & 6 \\ \hline
\end{array}
\]
Set 
\[
\gamma(G,r) = \left\{\begin{array}{ll} 
\dim G_{[r]} & \mbox{if $r=p$ or $r \in\{2,3\}$} \\
\ell(G) & \mbox{otherwise} 
\end{array}\right.
\]
where $\ell(G)$ is defined as follows (here $\delta_{i,j}$ is  the familiar Kronecker delta):
\[
\begin{array}{cccccccc} \hline
G & E_8 & E_7 & E_6 & F_4 & G_2 & D_4 & B_2 \\
\ell(G) & 200 & 100 & 58 & 40 & 10 & 24-2\delta_{5,r} & 8 \\ \hline
\end{array}
\]

\begin{theorem}\label{t:main6}
Let $G$ be a simply connected simple algebraic group as in \eqref{e:gro} and let $G(q)$ be a finite quasisimple exceptional group of Lie type over $\mathbb{F}_q$, where $q=p^f$. Let $r$ be a prime divisor of $|G(q)/Z(G(q))|$. Then the following hold:
\begin{itemize}\addtolength{\itemsep}{0.2\baselineskip}
\item[{\rm (i)}] $\mathcal{C}(G,r,q) \geqs \gamma(G,r)$.
\item[{\rm (ii)}] Let $g_r \in G$ be an element of order $r$ modulo $Z(G)$ with $\dim g_r^G \geqs \gamma(G,r)$ and let $M$ be a positive dimensional maximal closed subgroup of $G$. Then either
\[
\alpha(G,M,g_r) < \frac{2+\delta_{2,r}}{5},
\]
or $G=D_4$, $r=3$ and either $M=A_2$ and $\a(G,M,g_3) = 2/5$, or $M \in \{B_3,C_3\}$ and $\a(G,M,g_3) = 3/7$.
\end{itemize}
\end{theorem}

\begin{remk}
Note that part (i) gives $\mathcal{C}(G,r,q) = \dim G_{[r]}$ if $r=p$ or $r \in \{2,3\}$. It is also worth noting that there are examples with $\mathcal{C}(G,r,q) < \dim G_{[r]}$. For instance, if $G=F_4$ then $\dim G_{[7]} = 44$ and one checks that $\mathcal{C}(G,7,2) = 42$ (see \cite[Table 9]{BLS}, for example).
\end{remk} 
 
In the special case recorded in part (ii) of Theorem \ref{t:main6}, where $G = D_4$ and $M \in \{B_3,C_3\}$, we find that $\a(G,M,g_2) = 3/7$ (see Proposition \ref{p:par3}(i)). As an immediate corollary we obtain the following result, which is essential for our main application. In the statement, the elements $g_r, g_s \in G$ satisfy the conditions in Theorem \ref{t:main6}(ii).

\begin{corol}\label{cor:1}
If $r$ and $s$ are prime divisors of $|G(q)/Z(G(q))|$ and $(r,s) \ne (2,2)$, then
\[
\a(G,M,g_r) + \a(G,M,g_s) < 1
\]
for all positive dimensional maximal closed subgroups $M$ of $G$.
\end{corol}

Finally, we combine Corollary \ref{cor:1} with \cite[Theorems 1 and 2]{GLLS} to establish the following result on the random generation of finite simple exceptional groups of Lie type. This settles a conjecture of Liebeck and Shalev for exceptional groups (see \cite[p.2336]{GLLS}). Additionally, in the special case $(r,s)=(2,3)$, it provides a different proof of \cite[Theorem 9]{GLLS}. 

\begin{theorem}\label{t:main7} 
Let $r$ and $s$ be primes, not both $2$, and let $G_i$ be a sequence of finite simple exceptional groups of Lie type of order divisible by $r$ and $s$ such that $|G_i| \to \infty$ as $i \to \infty$. Then $\mathbb{P}_{r,s}(G_i) \to 1$ as $i \to \infty$. 
\end{theorem} 

In a sequel, we will establish similar results on the topological generation of classical algebraic groups. 

\vs

To conclude the introduction, let us say a few words on the layout of the paper. In Section \ref{s:top}, we consider the topological generation of simple algebraic groups in a general setting and we prove Theorems \ref{t:density} and \ref{t:dense2}, as well as Corollary \ref{c:gmt}. In Section \ref{ss:pt3} we give a short proof of Theorem \ref{t:main3} and then the remainder of Section \ref{s:fp} is devoted to deriving upper bounds on the dimensions of fixed point spaces arising in the action of an exceptional algebraic group on a coset variety (see Section \ref{ss:fps}). This is the most technical part of the paper, which culminates in a proof of Theorem \ref{t:main4}. Finally, in Section \ref{s:random} we prove Theorem \ref{t:main6} and we use it to establish Theorem \ref{t:main7}, which is our main result on the random generation of the finite simple exceptional groups of Lie type.

\section{Topological generation}\label{s:top}

In this section we prove Theorems \ref{t:density} and \ref{t:dense2}. Throughout, $G$ will denote a simply connected simple algebraic group over an algebraically closed field $k$ of characteristic $p \geqs 0$.  In addition, $\Omega$ is an irreducible subvariety of $G^t$ with $t \geqs 2$ and we will refer to the sets $\Delta$, $\Delta^+$ and $\Lambda$ defined in \eqref{e:deltap}. 

We first require an elementary lemma (see \cite[Theorem 5.1]{FGG}). In the statement, $F_t$ denotes the free group on $t$ generators. 

\begin{lem} \label{l:words}  
View $G \leqs \GL(V)$, where $V$ is a finite dimensional vector space over $k$. Observe that each  $x = (x_1, \ldots, x_t) \in G^t$ gives rise to a representation
$\rho_x :F_t \rightarrow \GL(V)$ by sending the $i$th generator of $F_t$ to $x_i$.  Let $V_x$ denote the corresponding module. 
For fixed $x, y \in G^t$, the following are equivalent: 
\begin{itemize}\addtolength{\itemsep}{0.2\baselineskip}
\item[{\rm (i)}]    $V_x/\rad(V_x) \cong V_y/\rad(V_y)$.  
\item[{\rm (ii)}]  For every $w \in F_t$, $w(x)$ and $w(y)$ have conjugate semisimple parts.
\end{itemize}
\end{lem}

In fact, by \cite[Corollary 5.3]{FGG}, we only need (ii) to hold for all words of length at most $2d^2$, where $d = \dim V$.

We can now prove Theorem \ref{t:density}. Indeed, the theorem follows by combining Lemmas \ref{l:words2} and \ref{l:words3} below. In order to state these results, we need to introduce some notation. 

For any word $w \in F_t$, let $\mathcal{S}(w)$ be a set of representatives of the conjugacy classes in $G$ of the semisimple parts of the elements in $\{w(x) \,:\, x \in \O\}$. First we handle the special case where $|\mathcal{S}(w)|=1$ for all $w \in F_t$.

\begin{lem}\label{l:words2}  
Let $\Omega$ be an irreducible subvariety
of $G^t$ with $t \geqs 2$ and assume $|\mathcal{S}(w)|=1$ for all $w \in F_t$.
\begin{itemize}\addtolength{\itemsep}{0.2\baselineskip}
\item[{\rm (i)}] If $p > 0$, then either $\Delta^+$ is empty, or $\Delta^+=\O$. 
\item[{\rm (ii)}]  If $p=0$, then $\Delta^+$ is an open subset of $\O$.  
\end{itemize}
\end{lem}

\begin{proof}  
We will use the notation from Lemma \ref{l:words}. By the previous lemma, the image of $G(x)$ in $\GL(V_x/\rad(V_x))$ is independent of $x \in \O$ (up to isomorphism) and so we may as well assume that all the images are finite (otherwise $\Delta^+=\O$).

If $p > 0$, then the unipotent radical of $G(x)$ has finite index in $G(x)$ and is therefore a finitely generated unipotent group, which implies that it is finite. Hence $G(x)$ is finite and we conclude that $\Delta^+$ is empty.

Now assume $p=0$. Here $G(x)$ is finite if and only if its unipotent radical is trivial. If the unipotent radical of $G(x)$ is trivial, then $|G(x)|=N$ for some fixed $N$ (since the quotient of $G(x)$ by its unipotent radical is independent of $x$, up to isomorphism). Since $\{x \in \O \,:\, |G(x)| \leqs N\}$ is closed, it follows that either $G(x)$ is finite for all $x \in \O$ (and thus $\Delta^+$ is empty), or the set of $x$ with $\dim G(x)>0$ is nonempty and open.
\end{proof}

\begin{rem}\label{r:ext}
Note that $|\mathcal{S}(w)|=1$ for all $w \in F_t$ if and only if the same property holds over any extension field $k'$ of $k$, whence the density of $\Delta^+$ in this situation is independent of $k'$. In particular, if $k$ is algebraic over a finite field, then $G(x)$ is always finite (and in fact $|G(x)|$ is absolutely bounded).  
\end{rem}

\begin{lem} \label{l:words3}  
Let $\Omega$ be an irreducible subvariety of $G^t$ with $t \geqs 2$ and assume that $|\mathcal{S}(w)|>1$ for some $w \in F_t$. Then $\Delta^+(k')$ is dense in $\O(k')$ for every algebraically
closed field $k'$ containing $k$ that is not algebraic over a finite field. 
\end{lem}

\begin{proof}  
In view of Remark \ref{r:ext}, we may assume that $k'=k$ is not algebraic over a finite field. 

Since $\O$ is an irreducible variety, the condition $|\mathcal{S}(w)|>1$ implies that $\mathcal{S}(w)$ is infinite. Moreover, the same conclusion holds if we view $G \leqs \GL(V)$ (that is, the set of semisimple parts of the elements in $\{w(x) \,:\, x \in \O\}$ meets infinitely many distinct conjugacy classes of ${\rm GL}(V)$). Indeed, the Weyl group controls fusion of semisimple elements, so each semisimple conjugacy class of $\GL(V)$ intersects $G$ in finitely many $G$-classes. Therefore, there are infinitely many distinct characteristic polynomials of the elements $w(x) \in {\rm GL}(V)$ as $x$ ranges over $\Omega$. Define $f_j: {\rm GL}(V) \to k$, where $f_j(g)$ is the $j$th coefficient of the characteristic polynomial of $g$. Then we may choose $j$ so that the restriction of $f_j$ to $\{w(x) \,:\, x \in \O\}$ is a non-constant function, whence the image of the restriction is cofinite in $k$. In particular, the set of $x \in \O$ such that $f_j(w(x))$ is not in the subfield of $k$ generated by roots of unity is dense (since we are assuming that $k$ is not algebraic over a finite field).  
This implies that $w(x) \in G(x)$ has infinite order for all $x$ in this dense subset of $\O$. The result follows. 
\end{proof}

Now Theorem \ref{t:density} follows by combining Lemmas \ref{l:words2} and \ref{l:words3}. At this point, we can also establish Corollary \ref{c:gmt}.  

\begin{proof}[Proof of Corollary \ref{c:gmt}]
Set $\O = C_1 \times \cdots \times C_t$ and let us assume $k$ is not algebraic over a finite field (of course, if the latter condition does not hold, then $\Delta^{+}$ is empty). Let $w \in F_t$ be the product of the generators of $F_t$ and define $\mathcal{S}(w)$ as above. By \cite[Theorem 1.1]{GMT}, either $\mathcal{S}(w)$ is infinite, or $t=2$ and $C_1C_2$ is a finite union of conjugacy classes, with $C_1$ and $C_2$ given explicitly. Therefore, aside from the special cases, Lemma \ref{l:words3} implies that $\Delta^+$ is dense in $\O$. 

To complete the argument, let us assume $t=2$ and $C_1,C_2$ are classes given in \cite[Theorem 1.1]{GMT}. If $p=0$ then at least one of the $C_i$ is unipotent and thus $\Delta^+ = \O$. Now assume $p>0$. Here $C_1$ and $C_2$ are both torsion classes, $\O$ is defined over the algebraic closure of a finite field and $|\mathcal{S}(w)|=1$ for all $w \in F_t$. By applying \cite[Corollary 5.14]{GMT}, we deduce that $G(x)$ is contained in a Borel subgroup of $G$ for all $x \in \O$. Therefore, since $C_1$ and $C_2$ are torsion classes, it follows that $G(x)/R_u(G(x))$ is a finite abelian group. Finally, as explained in the proof of Lemma \ref{l:words2}, we note that the unipotent radical of $G(x)$ is finite, so $G(x)$ is finite and we conclude that $\Delta^+$ is empty.
\end{proof}

Finally, we turn to Theorem \ref{t:dense2}. Here we need one more preliminary result. Recall that $\mathcal{M}$ is the set of positive dimensional maximal closed subgroups of $G$ and $\L$ is the set of $x \in \Omega$ such that $G(x)$ is not contained in such a subgroup. The next lemma can also be deduced from the stronger result \cite[Theorem 11.7]{GT}. 

\begin{lem}\label{l:prel}  
Set $\O = G^t$ with $t \geqs 2$ and assume $k$ is not algebraic over a finite field. Then there exists a nonempty open subset $\Gamma$ of $\O$ such that $\Delta \subseteq  \Gamma \subseteq \Lambda$.
\end{lem}

\begin{proof}   
We will give two different proofs.

The first proof uses the fact that there are only finitely many conjugacy classes of positive dimensional maximal closed subgroups of $G$ (see \cite[Corollary 3]{LS04}).   Let $M_1, \ldots, M_s$ be representatives of the distinct conjugacy classes of such subgroups and let $X_i$ be the closure of the image of the morphism $G \times M_i^t  \rightarrow G^t$ given by simultaneous conjugation. Since $t \geqs 2$ and each fiber has dimension at least $\dim M_i$, we see that this morphism is not dominant. Therefore, the union $X = \bigcup_{i}X_i$ is a proper closed subset of $G^t$ and thus $\Gamma := G^t \setminus X$ is a nonempty open subset of $\O$ contained in $\L$. Clearly, if $x \in \Delta$, then $x$ is not in $X$ and thus $\Delta \subseteq \Gamma$.

For the second proof, we produce a finite set $\mathcal{S}$ of irreducible $kG$-modules  with the property that each $M \in \mathcal{M}$ is reducible on some module in 
$\mathcal{S}$. First observe that 
\begin{equation}\label{e:gammadef}
\Gamma := \{ x \in \O \,:\, \mbox{$G(x)$ is irreducible on each $V \in \mathcal{S}$}\}
\end{equation}
is an open subset of $\O$ (see \cite[Lemma 11.1]{GT}), which is contained in $\L$ by construction. Note that in positive characteristic, $\Gamma$ will also contain all $x$ with $G(x) \cong G(q)$, as long as these subgroups act irreducibly on the modules in $\mathcal{S}$. Again, $\Delta$ is clearly
contained in $\Gamma$ and we know that $\Delta$ is nonempty (see \cite{gurnato}), whence $\Gamma$ is also nonempty. 

In almost all cases we can take $\mathcal{S} = \{V_1\}$, where $V_1$ is a nontrivial irreducible composition factor of the adjoint module. Excluding the special cases dealt with below, either $\dim V_1 \geqs \dim G - 1$, or $G = D_n$, $n \geqs 4$ and $\dim V_1  \geqs  \dim G - 2$. If $M$ is a positive dimensional closed subgroup of $G$, then either $M$ is contained in a proper parabolic subgroup or $\dim M < \dim V_1$. Since the Lie algebra of the connected component of $M$ is $M$-invariant, we deduce that $M$ acts reducibly on $V_1$. 

This argument applies unless $p=2$ and $G$ is of type $A_1$, $F_4$, $B_n$ or $C_n$, or $p=3$ and $G=G_2$. First assume $G = A_1$ with $p=2$. Let $\rho:G \to {\rm GL}(V_2)$ be the representation afforded by the natural module $V_2$ and let $\s$ be the standard Frobenius morphism of $G$ with $G_{\s} = {\rm SL}_{2}(2)$. Then we can take $\mathcal{S} = \{V_2 \otimes V_2^{(2)}\}$, where $V_2^{(2)}$ is the $kG$-module corresponding to the representation $\rho\s$ of $G$. 
Since each subgroup in $\mathcal{M}$ is either a Borel subgroup or the normalizer of a torus, the result follows in this case. If $G=F_4$, we take $\mathcal{S}$ to consist of the two $26$-dimensional modules and for $(G,p)=(G_2,3)$ we take the two $7$-dimensional modules.  By \cite{LSadj}, no subgroup in $\mathcal{M}$ is irreducible on both of these modules. 

Next assume $G=C_n$ and $p=2$. If $n=2$, we set $\mathcal{S} = \{V_1 \otimes V_1^{(2)}, V_2 \otimes V_2^{(2)}\}$ where the $V_i$ are the two $4$-dimensional fundamental $G$-modules. Any maximal positive dimensional reductive subgroup of $G$ is either $A_1A_1.2$ (two classes) or the normalizer of a torus and the result follows by inspection.  

Now assume $n \geqs 3$ and let $V$ be the natural module for $G$.   Take $\mathcal{S}=\{V_1, V_2\}$ where 
 $V_1$ is the nontrivial irreducible composition factor of $\L^2(V)$ of dimension $n(2n-1) - \delta$ (with $\delta \in \{1,2\}$) and $V_2$ is the (restricted) Steinberg module. Suppose that $M \in \mathcal{M}$ is irreducible on $V_1$. By \cite[Theorem 4.6]{GT}, it follows that $M = D_n.2$, or $n=3$ and $M=G_2$ (there is a small gap in the proof of \cite[Theorem 4.6]{GT} when $M$ is almost simple, but the conclusion is easily deduced from the main theorems of \cite{BT1,BT2,Se}). By \cite[Lemma 9.2]{GT}, $D_n.2$ is reducible on $V_2$. If $n=3$, then $\dim V_2= 512$ and trivially we see that $G_2$ has no irreducible representations of that dimension (see \cite[Table A.49]{Lu}, for example). The result follows.   

Finally, let us observe that $B_n$ is isogenous to $C_n$, so the previous argument  also handles the remaining $G=B_n$ case.  
\end{proof}  

We can now prove Theorem \ref{t:dense2}.

\begin{proof}[Proof of Theorem \ref{t:dense2}]  
Let $\Gamma$ be the nonempty open subset of $G^t$ in Lemma \ref{l:prel} and note that $\Gamma' := \Gamma \cap \Omega$ is an open subset of $\Omega$ with $\Delta \subseteq \Gamma' \subseteq \L$. Also recall that $\Delta = \L \cap \Delta^+$. 

Clearly, (ii) implies (i), and it also implies that $\Gamma'$ is a nonempty open subset of $\O$. It remains to show that (i) implies (ii). Suppose (i) holds, so $\Delta(k')$ is nonempty for some field extension $k'$ of $k$. Then $\Delta^+(k')$ is nonempty and thus $\Delta^+$ is dense in $\O$ by Theorem \ref{t:density}. In addition, $\Gamma'(k')$ is a nonempty open subset of $\O(k')$ and since $\Gamma'$ is defined over $k$, we deduce that $\Gamma' = \Gamma'(k)$ is also nonempty. It follows that $\Delta = \L \cap \Delta^+ = \Gamma' \cap \Delta^+$ is dense and (ii) holds. 
\end{proof}

To conclude this section, we consider the case that $k$ is algebraic over a finite field.  Of course, $G$ is locally finite in this situation, so $G(x)$ is always finite. The following result will play a role in the proof of Theorem \ref{t:main7} (see Section \ref{ss:thm7}).

\begin{thm}\label{t:algebraic}  
Let $G$ be a simple algebraic group over $k$, where $k$ is the algebraic closure of a finite field of characteristic $p$. Let $\Omega$ be an irreducible subvariety of $G^t$ with $t \geqs 2$. Then the following are equivalent:
\begin{itemize}\addtolength{\itemsep}{0.2\baselineskip}
\item[{\rm (i)}] $G(x) = G(k')$ for some $x  \in \O(k')$ and field extension $k'$ of $k$.
\item[{\rm (ii)}] For any fixed integer $d$, there exists $x \in \Omega$ with $G(x) \cong G(q)$ for some $p$-power $q > d$.
\end{itemize}
\end{thm}

\begin{proof}   
First observe that if $|\mathcal{S}(w)|=1$ for all words $w \in F_t$, then the proof of 
Lemma \ref{l:words} shows that $G(x)$ has finite bounded order for all $x \in \O$. Clearly, this cannot happen if (i) or (ii) hold, so in both cases we see that there is a word $w \in F_t$ with $|\mathcal{S}(w)| > 1$. By the irreducibility of $\O$, it follows that $\mathcal{S}(w)$ is infinite and thus $w(x)$ can have arbitrarily large order.   

Let $\Gamma$ be the open subset of $G^t$ described in the (second) proof of Lemma \ref{l:prel} (see \eqref{e:gammadef}) and set $\Gamma' = \Gamma \cap \O$ as in the proof of Theorem \ref{t:dense2}. If (i) holds, then $\Gamma'(k')$ is nonempty, which implies that $\Gamma$ is nonempty since it is defined over $k$. Similarly, $\Gamma$ is nonempty if (ii) holds. To see this, notice that for all $q$ sufficiently large, $G(q)$ is irreducible on each of the $kG$-modules in the set $\mathcal{S}$ described in the (second) proof of Lemma \ref{l:prel}.

Suppose (ii) holds and let $k'$ be an algebraically closed field containing $k$ that is not algebraic over a finite field. Since $|\mathcal{S}(w)| > 1$ for some $w \in F_t$, by applying Lemma \ref{l:words3} we deduce that $\Delta^+(k') $ is dense in $\Omega(k')$ and thus $\Delta(k') = \L(k') \cap \Delta^+(k')$ is nonempty.  Therefore, (ii) implies (i).

Now assume (i) holds. For any integer $n$, we note that $\{x \in \O \,:\, |G(x)| > n\}$ is a nonempty open subset and therefore meets $\L$. By definition of $\L$, if $x$ is in the intersection then $G(x)$ is not contained in a proper positive dimensional closed subgroup of $G$. Therefore, by taking $n$ sufficiently large, we deduce that $G(x)$ is a subfield subgroup of $G$ (possibly twisted) and (ii) follows. Here the fact
that any sufficiently large finite subgroup of $G$ that is not contained in a proper positive dimensional closed subgroup is a subfield subgroup follows by \cite{As} for classical groups and by \cite{LS01} for exceptional groups. It also follows from a result of Larsen and Pink \cite{LP}. 
\end{proof} 

\section{Fixed point spaces for exceptional algebraic groups}\label{s:fp}

In this section we prove Theorems \ref{t:main3} and \ref{t:main4}, which combine to give Theorem \ref{t:main5}. We adopt the notation introduced in Section \ref{s:intro}. In particular, for $g \in G$ and a coset variety $X=G/M$, we write $X(g)$ for the variety of fixed points of $g$ on $X$ and we define
\[
\a(g) = \a(G,M,g) = \frac{\dim X(g)}{\dim X}.
\]
Let us also recall \cite[Proposition 1.14]{LLS}, which states that
\begin{equation}\label{e:lls}
\dim X(g) = \dim X - \dim g^G + \dim (g^G \cap M)
\end{equation}
for all $g \in M$ (of course, $X(g)$ is nonempty if and only if $g^G \cap M$ is nonempty).

\subsection{Proof of Theorem \ref{t:main3}}\label{ss:pt3}

Let $M$ be a positive dimensional maximal closed subgroup of $G$ and set $X=G/M$ and $\Omega = C_1 \times \cdots \times C_t$. If $C_i \cap M$ is empty for some $i$, then for all $x \in \O$ the subgroup $G(x)$ is not contained in a conjugate of $M$; so let us assume each $C_i$ meets $M$ and consider the variety
\[
Y= \{(g_1, \ldots, g_t, x) \,:\, g_i \in C_i, \, x \in X(g_i), \, i = 1, \ldots, t \}.
\]
By projecting $Y$ onto $X$ by sending each $(t + 1)$-tuple to its last component, and noting that all fibers of this projection have the same dimension, we see that 
\[
\dim Y = \dim X + \sum_{i=1}^{t} \dim (C_i \cap M).
\] 
By applying \eqref{e:lls}, 
\[
\sum_{i=1}^{t} \dim C_i = \sum_{i=1}^{t} \dim(C_i \cap M)  +  \sum_{i=1}^{t} (\dim X  - \dim X(g_i))
\]
and we deduce that
\[
\dim Y  =  \sum_{i=1}^{t}\dim C_i  - (t-1) \dim X + \sum_{i=1}^{t}\a(G,M,g_i)\dim X.
\]
Therefore, the hypothesis implies that $\dim Y < \sum_i \dim C_i$ and thus the projection from $Y$ into $\Omega$ is not dominant. It follows that the set of $x \in \Omega$ such that $G(x)$ is contained in a conjugate of $M$
is contained in a proper closed subset of $\Omega$. 

Theorem \ref{t:main3} now follows since $G$ has only finitely many conjugacy classes of positive dimensional  maximal closed  subgroups (see \cite[Corollary 3]{LS04}).

\subsection{Fixed point spaces for exceptional algebraic groups}\label{ss:fps}

For the remainder of Section \ref{s:fp} we will focus on the proof of Theorem \ref{t:main4}.

Let $G$ be a simple algebraic group and let $g$ be a non-central element of $G$. Set  $X=G/M$, where $M$ is a positive dimensional maximal closed subgroup of $G$. Write $g=su$, where $s$ is the semisimple part of $g$ and $u$ is the unipotent part. Then for $h \in G$ we have $g \in M^h$ if and only if both $s \in M^h$ and $u \in M^h$, so $X(g) = X(s) \cap X(u)$. Therefore, for the purposes of proving Theorem \ref{t:main4}, we may assume that $g$ is either semisimple or unipotent. In addition, if the order of $g$ is finite then we may assume $g$ has prime order since $X(g) \subseteq X(g^n)$ for every positive integer $n$. Finally, suppose $g \in G$ is semisimple of infinite order. Here we observe that 
$X(g) \subseteq X(h)$ for all elements $h$ in the closure of $\langle g \rangle$ and we note that this subgroup contains a positive dimensional torus and therefore elements of order $2$ or $3$.  

This shows that in order to prove Theorem \ref{t:main4} we may assume $g \in \mathcal{P}$, which is the set of elements of prime order in $G$ (as well as all nontrivial unipotent elements if $p=0$). Our main result is as follows, which immediately implies Theorem \ref{t:main4}.

\begin{thm}\label{t:mainex}
Let $G$ be a simply connected simple algebraic group of exceptional type over an algebraically closed field of characteristic $p \geqs 0$, let $M$ be a positive dimensional maximal closed subgroup of $G$ and let $g \in \mathcal{P}$ be non-central. Then 
\[
\a(G,M,g) \geqs \frac{2}{3}
\]
if and only if $(G,M,g)$ is one of the cases recorded in Table \ref{tab:main}.
\end{thm}

\begin{rem}\label{r:tab}
In the third column of Table \ref{tab:main}, $u_{\a}$ denotes a long root element, $u_{\b}$ is a short root element (with $p=2$ if $G=F_4$) and $t$ is a $B_4$-involution in $F_4$ (that is, the centralizer of $t$ in $F_4$ is a group of type $B_4$). In addition, we write $T_i$ for an $i$-dimensional torus and $P_i$ denotes the maximal parabolic subgroup of $G$ corresponding to deleting the $i$-th node in the Dynkin diagram of $G$, labelled as in Bourbaki \cite{Bou}. We also adopt the notation $\tilde{Y}$ for a subgroup of type $Y$ that is generated by short root subgroups (for example, $F_4$ has a subgroup $\tilde{D}_4$).
\end{rem}

\begin{cor}\label{c:1}
We have $\a(G,M,g) \geqs \frac{2}{3}$ only if $g$ is a long root element, or a short root element if $(G,p) = (F_4,2)$ or $(G_2,3)$, or a $B_4$-involution if $G=F_4$ and $p \ne 2$.
\end{cor}

{\small
\renewcommand{\arraystretch}{1.05}
\begin{table}
\[
\begin{array}{llcc} \hline
G & M^0 & g &  \a(G,M,g)  \\ \hline
E_8 & P_8 & u_{\a} & 15/19 \\
& A_1E_7 & u_{\a} & 11/14 \\
& P_7 & u_{\a} & 65/83 \\
& A_2E_6 & u_{\a} & 7/9  \\
& P_6 & u_{\a} & 75/97 \\
& D_4^2 \, (p=2) & u_{\a} & 37/48  \\
& P_1, G_2F_4  & u_{\a} & 10/13 \\
& P_3 & u_{\a} & 75/98 \\
& P_4 & u_{\a} & 81/106 \\
& T_8 \, (p=2) & u_{\a} & 61/80 \\
& P_2 & u_{\a} & 35/46 \\
& A_1G_2^2 \, (p \ne 2) & u_{\a} & 165/217  \\
& P_5 & u_{\a} & 79/104 \\
& D_8, A_8, A_4^2, A_2^4, A_1^8, D_4^2 \, (p \ne 2) & u_{\a} & 3/4  \\
& & & \\
E_7  & P_7, E_6T_1 & u_{\a} & 7/9 \\
&  A_1F_4 & u_{\a} & 10/13 \\
& P_6 & u_{\a} & 16/21 \\
& P_1 & u_{\a} & 25/33 \\
& A_1D_6, A_1^3D_4 & u_{\a} & 3/4 \\
& P_3 & u_{\a} & 35/47 \\
& P_5 & u_{\a} & 37/50 \\
& P_2, T_7 \, (p=2) & u_{\a} & 31/42 \\
& P_4 & u_{\a} & 39/53 \\
& A_2A_5 & u_{\a} & 11/15 \\
& A_7, A_1^7, G_2C_3 & u_{\a} &  5/7 \\
& & & \\
E_6  &  F_4 & u_{\a} & 10/13 \\
& P_1, P_6, D_4T_2 & u_{\a} & 3/4 \\
& P_3, P_5 & u_{\a} & 18/25 \\
& P_2, A_2G_2 & u_{\a} & 5/7 \\
& T_6 \, (p=2) & u_{\a} & 17/24 \\
& A_1A_5 & u_{\a} & 7/10 \\
& P_4 & u_{\a} & 20/29 \\
& A_2^3, C_4\, (p \ne 2) & u_{\a} & 2/3 \\
& & & \\
F_4  & B_4, D_4 & u_{\a} & 3/4 \\
& C_4 \,(p=2), \tilde{D}_4\,(p=2)  & u_{\b} & 3/4 \\
& P_1 & u_{\b}, t & 11/15 \\
& P_4  & u_{\a} & 11/15 \\
& A_1C_3 & t & 5/7 \\
& A_1G_2 & u_{\a} & 5/7 \\
& P_2 & u_{\b}, t & 7/10 \\
& P_3 & u_{\a} & 7/10 \\
& P_1 & u_{\a} & 2/3 \\
& P_4 & u_{\b},t & 2/3 \\
& A_2\tilde{A}_2 & u_{\a}, u_{\b}, t & 2/3 \\
& & & \\
G_2  & A_2 & u_{\a} & 2/3 \\
& \tilde{A}_2\, (p=3) & u_{\b} & 2/3 \\ \hline
\end{array}
\]
\caption{The cases in Theorem \ref{t:mainex} with $\a(G,M,g) \geqs 2/3$}
\label{tab:main}
\end{table}
\renewcommand{\arraystretch}{1}}

The proof of Theorem \ref{t:mainex} will rely heavily on the following theorem of Liebeck and Seitz \cite{LS04}, which classifies the positive dimensional maximal closed  subgroups of exceptional algebraic groups. 

\begin{thm}\label{t:ls}
Let $G$ be a simple algebraic group of exceptional type over an algebraically closed field $k$ of characteristic $p \geqs 0$ and let $M$ be a positive dimensional maximal closed subgroup of $G$. Then one of the following holds:
\begin{itemize}\addtolength{\itemsep}{0.2\baselineskip}
\item[{\rm (i)}] $M$ is a parabolic subgroup;
\item[{\rm (ii)}] $M^0$ is a reductive subgroup of maximal rank, as in Table \ref{tab:mr};
\item[{\rm (iii)}] $G = E_7$, $p \ne 2$ and $M = (2^2 \times D_4).S_3$;
\item[{\rm (iv)}] $G = E_8$, $p \ne 2,3,5$ and $M = A_1 \times S_5$;
\item[{\rm (v)}] $M^0$ is as in Table \ref{tab:nmr}.
\end{itemize}
\end{thm}

\begin{rem}
Note that in each of the cases in part (v) of Theorem  \ref{t:ls}, the component group $M/M^0$ is either trivial or of order $2$, as given in the fourth column of \cite[Table 10.1]{LS04}.
\end{rem}

\renewcommand{\arraystretch}{1.1}
\begin{table}
\begin{center}
\[
\begin{array}{lll}\hline
G & M^0 & M/M^0 \\ \hline
E_8 & D_8, A_1E_7, A_8, A_2E_6, A_4^2, D_4^2, A_2^4, & 1, 1, Z_2, Z_2, Z_4, S_3 \times Z_2, {\rm GL}_{2}(3), \\
& A_1^8, T_8 & {\rm AGL}_{3}(2), W(E_8) \\
E_7 & A_1D_6, A_7, A_2A_5, A_1^3D_4, A_1^7, E_6T_1, T_7 & 1, Z_2, Z_2, S_3, {\rm GL}_{3}(2), Z_2, W(E_7) \\
E_6 & A_1A_5, A_2^3, D_4T_2, T_6 & 1, S_3, S_3, W(E_6) \\
F_4\,  (p \ne 2) & B_4, D_4, A_1C_3, A_2\tilde{A}_2 & 1, S_3, 1, Z_2 \\
F_4\,  (p=2) & B_4, C_4, D_4, \tilde{D}_4, A_2\tilde{A}_2 & 1, 1, S_3, S_3, Z_2 \\
G_2 & A_1\tilde{A}_1, A_2, \tilde{A}_2\, (p=3) & 1, Z_2, Z_2 \\ \hline
\end{array}
\]
\caption{The possibilities for $M^0$ in Theorem \ref{t:ls}(ii)}
\label{tab:mr}
\end{center}
\end{table}
\renewcommand{\arraystretch}{1}

\renewcommand{\arraystretch}{1.1}
\begin{table}
\begin{center}
\[
\begin{array}{cll}\hline
G & \mbox{$M^0$ simple} & \mbox{$M^0$ not simple} \\ \hline
E_8 & A_1 \, (\mbox{$3$ classes, $p \geqs 23,29,31$}), B_2\, (p \geqs 5) & A_1A_2\, (p \ne 2,3),  A_1G_2^2 \, (p \ne 2), G_2F_4 \\
E_7 & A_1 \, (\mbox{$2$ classes, $p \geqs 17,19$}),  A_2\, (p \geqs 5) & A_1^2\, (p \ne 2,3),  A_1G_2 \, (p \ne 2), A_1F_4, G_2C_3 \\
E_6 & A_2 \, (p \ne 2,3),  G_2\, (p \ne 7), \, C_4\, (p \ne 2),  F_4 & A_2G_2 \\
F_4 & A_1 \, (p \geqs 13), G_2\, (p=7)  & A_1G_2\, (p \ne 2) \\
G_2 & A_1 \, (p \geqs 7) & \\ \hline
\end{array}
\]
\caption{The possibilities for $M^0$ in Theorem \ref{t:ls}(v)}
\label{tab:nmr}
\end{center}
\end{table}
\renewcommand{\arraystretch}{1}

For $g \in G$, it will be convenient to write $\a(g) = \a(G,M,g)$ if the context is clear. Similarly, we also define
\begin{equation}\label{e:bx}
\b(g) = \b(G,M,g) = \dim X - \dim X(g).
\end{equation}  

\begin{rem}\label{r:field}
We claim that in order to prove Theorem \ref{t:mainex} we may assume that $p > 0$ and $k$ is algebraic over $\mathbb{F}_p$. First assume that $p > 0$.  Then each $g \in \mathcal{P}$ has prime order and is therefore defined over a finite field.  Similarly, both $G$ and every maximal closed subgroup of $G$ are defined over a finite field (this is true even for the finite maximal subgroups, but we do not require this) and it follows that $\alpha(g)$ does not change if we replace $k$ by the algebraic closure of 
$\mathbb{F}_p$. Finally, the claim when $p=0$ follows by a standard compactness argument (the description of the unipotent conjugacy classes in characteristic $0$ is the same as in any prime characteristic $p$ that is good for $G$).  
\end{rem}
 
There is an extensive literature on conjugacy classes in algebraic groups of exceptional type and our notation is fairly standard. In particular, we will adopt the labelling of unipotent classes from \cite{LS_book} and we refer the reader to \cite{Lub} for detailed information on semisimple conjugacy classes and the corresponding centralizers. Our notation for modules is also standard. In particular, for a connected reductive algebraic group $H$ we will write ${\rm Lie}(H)$ for the adjoint module and $V_H(\l)$ (or just $V(\l)$ or $\l$ if the context is clear) for the rational irreducible $H$-module with highest weight $\lambda$ (and we will label the fundamental dominant weights $\l_i$ for $H$ in the usual manner, following \cite{Bou}). Similarly, $W_H(\l)$ is the Weyl module with highest weight $\l$ and the trivial module will be denoted by $0$. We write $W(H)$ for the Weyl group of $H$. Finally, if $M$ is a positive dimensional maximal closed non-parabolic subgroup of $G$ then we will often work with the restriction of $V$ to the connected component $M^0$, denoted by $V{\downarrow}M^0$, where $V$ is typically the adjoint module or minimal module for $G$. The tables in \cite[Chapter 12]{Thomas} provide a  convenient reference for the composition factors of $V{\downarrow}M^0$.

The proof of Theorem \ref{t:mainex} is organised as follows. First, in Section \ref{s:parab}, we study the case where $M$ is a maximal parabolic subgroup. The reductive maximal rank subgroups are treated in Section \ref{s:mr} and the proof is completed in Section \ref{s:rem}, where the remaining maximal subgroups arising in parts (iii), (iv) and (v) of Theorem \ref{t:ls} are handled. Finally, in Section \ref{ss:generic} we use Theorem \ref{t:mainex} to establish Theorem \ref{t:generic}, which is our main result on generic stabilizers.

\subsection{Parabolic subgroups}\label{s:parab}

Let $G$ be a simply connected simple algebraic group of exceptional type over an algebraically closed field $k$ of characteristic $p \geqs 0$ and let $M = P_i$ be a maximal parabolic subgroup of $G$ (here $i$ corresponds to the $i$-th node in the Dynkin diagram of $G$, labelled as in Bourbaki \cite{Bou}). Set $X = G/M$ and note that the dimension of $X$ is given in Table \ref{t:parab}.

\begin{table}
$$\begin{array}{r|cccccccc}
 & P_{1} & P_{2} & P_{3} & P_{4} & P_{5} & P_{6} & P_{7} & P_{8} \\ \hline
E_{8} & 78 & 92 & 98 & 106 & 104 & 97 & 83 &  57 \\
E_{7} & 33 & 42 & 47 & 53 & 50 & 42 & 27 & \\
E_{6} & 16  & 21 & 25 & 29 & 25 & 16 & & \\
F_{4} & 15 & 20 & 20 & 15 & & & & \\
G_{2} & 5 & 5 & & & & & & \\
\end{array}$$
\caption{$G$ exceptional, $\dim G/P_i$}
\label{t:parab}
\end{table}

We start by considering root elements.

\begin{lem}\label{l:p1}
Let $M = P_i$ be a maximal parabolic subgroup of $G$.
\begin{itemize}\addtolength{\itemsep}{0.2\baselineskip}
\item[{\rm (i)}] If $g \in G$ is a long root element, then $\a(g)$ is given in Table \ref{t:parab2}.
\item[{\rm (ii)}] If $G= F_4$ or $G_2$ and $g \in G$ is a short root element, then $\a(g)$ is given in Table \ref{t:parab3}.
\end{itemize} 
\end{lem}

\begin{proof}
This follows immediately from \cite[Theorem 2(I)(a)]{LLS}, which gives $\b(g)$. 
\end{proof}

\begin{table}
\[
\begin{array}{r|cccccccc}
 & P_{1} & P_{2} & P_{3} & P_{4} & P_{5} & P_{6} & P_{7} & P_{8} \\ \hline
E_{8} & 10/13  & 35/46 & 75/98 & 81/106 & 79/104 & 75/97 & 65/83 & 15/19 \\
E_{7} & \hspace{1.7mm} 25/33 & 31/42 & 35/47 & 39/53 & 37/50 & 16/21 & 7/9 & \\
E_{6} & 3/4  & 5/7 & 18/25 & 20/29 & 18/25 & 3/4 & & \\
F_{4} & 2/3 & 13/20 & 7/10 & 11/15 & & & & \\
G_{2} & 2/5 & 3/5 & & & & & & \\
\end{array}
\]
\caption{$\a(g)$, $M = P_i$, $g$ long root element}
\label{t:parab2}
\end{table}

\begin{table}
\[
\begin{array}{r|cccc} 
& P_1 & P_2 & P_3 & P_4 \\ \hline
F_4 & (9+2\delta_{2,p})/15 & (11+3\delta_{2,p})/20 & (11+2\delta_{2,p})/20 & (9+\delta_{2,p})/15 \\ 
G_2 & (2+\delta_{3,p})/5 & 2/5
\end{array}
\]
\caption{$\a(g)$, $M = P_i$, $g$ short root element}
\label{t:parab3}
\end{table}

Next we handle the remaining unipotent elements.

\begin{lem}\label{l:p2}
Suppose $M = P_i$ and $g \in G$ is a unipotent element, which is neither a long nor short root element. Then $\a(g) < \frac{2}{3}$.
\end{lem}

\begin{proof}
Once again we use \cite[Theorem 2(I)(a)]{LLS}, which provides a lower bound on $\b(g)$ (see \cite[Tables 7.1, 7.2]{LLS}). It is easy to check that this bound gives $\a(g) < 2/3$ in every case.
\end{proof}

Finally, we consider semisimple elements.

\begin{lem}\label{l:p3}
Suppose $M = P_i$ and $g \in G$ is a semisimple element. Then $\a(g) \geqs \frac{2}{3}$ if and only if $G = F_4$, $M \in \{P_1, P_2, P_4\}$, $p \ne 2$ and $g \in G$ is a $B_4$-involution, in which case 
\[
\a(g) = \left\{\begin{array}{ll}
11/15 & \mbox{if $M=P_1$} \\
7/10 & \mbox{if $M=P_2$} \\
2/3 & \mbox{if $M=P_4$.} 
\end{array}\right.
\]
\end{lem}

\begin{proof}
We apply the lower bound on $\b(g)$ given in \cite[Theorem 2(I)(b)]{LLS} (see \cite[Table 7.3]{LLS}). One checks that this gives $\a(g) < 2/3$, unless $(G,M,g)$ is one of the three cases identified in the statement of the lemma. Here \cite[Theorem 2(I)(b)]{LLS} implies that $\a(g)$ is at most the given value, so in order to complete the proof of the lemma it remains to show that equality holds in each case. 

To do this, we can argue as follows (as noted in Remark \ref{r:field}, we may, and will,  assume that $k = \bar{\mathbb{F}}_p$). Let $q$ be a $p$-power and consider the finite group $G(q)=F_4(q)$ acting on the cosets of the corresponding maximal parabolic subgroup $P_i(q)$ of $G(q)$. Let $\chi_i$ be the associated permutation character. If $g \in G(q)$ is semisimple, then \cite[Corollary 3.2]{LLS2} gives an expression for $\chi_i(g)$ which we can use to compute the precise number of fixed points of a $B_4$-type involution $g$:
\[
\begin{array}{cl} \hline
i & \chi_i(g) \\ \hline
1 & (q^4+q^2+1)(q^4+1)(q^2+1)(q+1) \\
2 & (q^4+q^2+1)(q^4+1)(q^2+1)^2(q+1)^2 \\
4 & (q^4+1)(q^3+2)(q^2+1)(q+1) \\ \hline
\end{array}
\]

Notice that in every case, $\chi_i(g)$ is a monic polynomial in $q$. Moreover, by Lang-Weil \cite{LW}, the degree of this polynomial is equal to the dimension of $X(g)$, where $X = G/P_i$. We conclude that $\dim X(g) = 11, 14, 10$ for $i=1,2,4$, respectively. The result follows.
\end{proof}

\subsection{Maximal rank subgroups}\label{s:mr}

Next we handle the reductive maximal rank subgroups listed in Table \ref{tab:mr}. In Table \ref{tab:mr2}, we adopt the notation used in Table \ref{tab:main}.

\begin{prop}\label{t:mr}
Let $G$ be a simply connected simple algebraic group of exceptional type, let $M$ be one of the maximal rank subgroups in Table \ref{tab:mr} and let $g \in \mathcal{P}$ be non-central. Then 
$\a(G,M,g) \geqs \frac{2}{3}$
if and only if $(G,M,g)$ is one of the cases recorded in Table \ref{tab:mr2}.
\end{prop}

\renewcommand{\arraystretch}{1.1}
\begin{table}[h]
\[
\begin{array}{llcc} \hline
G & M^0 & g &  \a(G,M,g)  \\ \hline
E_8 & A_1E_7 & u_{\a} & 11/14 \\
& A_2E_6 & u_{\a} & 7/9  \\
& D_4^2 \, (p=2) & u_{\a} & 37/48  \\
& T_8 \, (p=2) & u_{\a} & 61/80 \\
& D_8, A_8, A_4^2, A_2^4, A_1^8, D_4^2 \, (p \ne 2) & u_{\a} & 3/4  \\
E_7 & E_6T_1 & u_{\a} & 7/9 \\
& A_1D_6, A_1^3D_4 & u_{\a} & 3/4 \\
& T_7 \, (p=2) & u_{\a} & 31/42 \\
& A_2A_5 & u_{\a} & 11/15 \\
& A_7, A_1^7 & u_{\a} &  5/7 \\
E_6 & D_4T_2 & u_{\a} & 3/4 \\
& T_6 \, (p=2) & u_{\a} & 17/24 \\ 
& A_1A_5 & u_{\a} & 7/10 \\
& A_2^3 & u_{\a} & 2/3 \\
F_4 & B_4, D_4 & u_{\a} & 3/4 \\
& C_4 \,  (p=2),  \tilde{D}_4 \, (p=2) & u_{\b} & 3/4 \\
& A_1C_3 & t & 5/7 \\
& A_2\tilde{A}_2 & u_{\a}, t & 2/3 \\
& A_2\tilde{A}_2\, (p=2)  & u_{\b} & 2/3 \\
G_2 & A_2 & u_{\a} & 2/3 \\
& \tilde{A}_2\, (p=3) & u_{\b} & 2/3 \\ \hline
\end{array}
\]
\caption{$M^0$ reductive and maximal rank, $\a(G,M,g) \geqs 2/3$}
\label{tab:mr2}
\end{table}
\renewcommand{\arraystretch}{1} 

\begin{lem}\label{l:mre8}
The conclusion to Proposition \ref{t:mr} holds if $G = E_8$.
\end{lem}

\begin{proof}
Let $g$ be an element in $\mathcal{P}$ and set $\a(g) = \a(G,M,g)$. Also define $\b(g)$ as in \eqref{e:bx} and let $V$ be the Lie algebra of $G$. The possibilities for $M^0$ are as follows:
\[
\begin{array}{cccccccccc} \hline
M^0 & D_8 & A_1E_7 & A_8 & A_2E_6 & A_4^2 & D_4^2 &  A_2^4 & A_1^8 & T_8 \\ 
\dim X & 128 & 112 & 168 & 162 & 200 & 192 & 216 & 224 & 240 \\ \hline
\end{array} 
\]
Without loss of generality, we may assume that $g \in M$.

If $g$ is semisimple, then the desired result follows from the lower bound on $\b(g)$ in \cite{LLS}. For example, if $M^0 = D_8$ then \cite[Theorem 2(II)(b)]{LLS} gives $\b(g) \geqs 56$, so $\dim X(g) \leqs 72$ and thus $\a(g) \leqs 9/16$.

For the remainder, we may assume $g$ is unipotent. If $M^0 = D_8$ or $A_1E_7$ then $M$ is connected and we can appeal to \cite{Law09}, where the $G$-class of each unipotent element in $M$ is determined. For example, suppose $M = A_1E_7$. If $g = u_{\a}$ is a long root element then $\dim g^G = 58$ and $g^G \cap M$ is a union of two $M$-classes (comprising long root elements in each simple factor), whence $\dim (g^G \cap M) = 34$ and \eqref{e:lls} yields $\dim X(g) = 88$. Therefore, $\a(g) = 11/14$. (Note that in this case we can also compute $\dim (g^G \cap M)$ directly by applying \cite[Proposition 1.13(ii)]{LLS}.) For all other unipotent elements we find that $\b(g) \geqs 40$ (with equality if $g$ is in the class labelled $A_1^2$), so $\dim X(g) \leqs 72$ and thus $\a(g) \leqs 9/14<2/3$. 

Next assume $M^0 = A_8$ or $A_2E_6$, so $M/M^0=Z_2$ in both cases. Suppose $g^G \cap (M \setminus M^0)$ is nonempty, so $p=2$. As explained in the proof of \cite[Proposition 5.11]{BGS}, we can work with the restriction of $V$ to $M^0$ to calculate the Jordan form of $g$ on $V$ and we can then inspect \cite[Table 9]{Lawunip} to identify the $G$-class of $g$. For example, suppose $M^0 = A_2E_6$ and $p=2$. There are two $M^0$-classes of involutions in $M \setminus M^0$, represented by $g_1$ and $g_2$, with $C_{M^0}(g_1) = A_1F_4$ and $C_{M^0}(g_2) = A_1C_{F_4}(u)$, where $u \in F_4$ is a long root element; whence $\dim (g_1^G \cap (M \setminus M^0)) = 31$ and $\dim (g_2^G \cap (M \setminus M^0)) = 47$. We calculate that $g_1$ is in the $G$-class $A_1^3$, and $g_2$ is in the class $A_1^4$,  so that $\dim g_1^G = 112$ and $\dim g_2^G = 128$. By inspecting \cite{Law09}, we get $\dim (g_1^G \cap M^0) = 40$ and $\dim (g_2^G \cap M^0) = 44$ and thus $\dim X(g_1) = 90$ and $\dim X(g_2) = 81$. In particular, if $g^G \cap (M \setminus M^0)$ is nonempty then $\a(g) \leqs 5/9$. Finally, if $g^G \cap M \subseteq M^0$ then the desired result quickly follows from the computations in \cite{Law09}.

The case $M^0 = A_4^2$ is very similar. Here $M/M^0 = Z_4$ and the fusion of unipotent classes in $M^0$ is determined in \cite{Law09}. This allows us to reduce to the case where $g^G \cap (M \setminus M^0)$ is nonempty. Here $p=2$ and $g$ acts as a graph automorphism on both $A_4$ factors, so $\dim (g^G \cap (M \setminus M^0)) = 28$. By considering $V{\downarrow}M^0$ (see \cite[Chapter 12]{Thomas}, for example), we calculate that $g$ has Jordan form $[J_2^{120},J_1^8]$ on $V$, in which case \cite[Table 9]{Lawunip} implies that $g$ is in the $A_1^4$ class. Therefore $\dim g^G = 128$ and $\dim (g^G \cap M) = 28$, which gives $\a(g)=1/2$.

To complete the proof of the lemma, we need to handle unipotent elements in the following cases
\[
M^0 \in \{D_4^2, A_2^4, A_1^8, T_8\}.
\]
Note that none of these cases are treated by Lawther in \cite{Law09}.

Suppose $M^0 = D_4^2$ and observe that $\dim (g^G \cap M) \leqs 48$. If $\dim g^G \geqs 114$, then $\b(g) \geqs 114-48 = 66$, so $\dim X(g) \leqs 126$ and thus $\a(g) \leqs 21/32<2/3$. Therefore, we may assume that $\dim g^G < 114$, in which case $g$ is in one of the classes $A_1, A_1^2$ or $A_1^3$. In the first two cases, we can appeal directly to the proof of \cite[Proposition 5.11]{BGS} to see that $\dim (g^G \cap M) = 10+4\delta_{2,p}$ and $20$ when $g$ is in $A_1$ and $A_1^2$, respectively. In particular, if $g=u_{\a}$ then $\dim X(g) = 144+4\delta_{2,p}$ and thus $\a(g)$ is $3/4$ if $p \ne 2$ and $37/48$ if $p=2$. Finally, suppose $g$ is in the $A_1^3$ class, so $\dim g^G = 112$. We claim that $\dim (g^G \cap M) \leqs 46$, so $\b(g) \geqs 66$ and $\a(g) \leqs 21/32$. Seeking a contradiction, suppose there exists $y \in g^G \cap M$ with $\dim y^M>46$. Then $y=y_1y_2 \in M^0$ and both $y_1$ and $y_2$ are regular, so $p \geqs 7$. By considering the restriction of $V$ to $M^0$ (see \cite[Chapter 12]{Thomas}) we deduce that the Jordan form of $y$ on $V$ has Jordan blocks of size $m \geqs 6$. But this is incompatible with the Jordan form of $g$ on $V$ (see \cite[Table 9]{Lawunip}), so we have reached a contradiction. This establishes the claim, which completes the analysis of this case.

Next assume $M^0 = A_2^4$, so $M/M^0 = {\rm GL}_{2}(3)$. If $\dim g^G \geqs 112$ then the trivial bound $\dim (g^G \cap M) \leqs 32$ yields $\dim X(g) \leqs 136$ and the result follows. Therefore, we may assume $g$ is in the class $A_1$ or $A_1^2$. If $g=u_{\a}$ then $\dim (g^G \cap M) = 4$ by \cite[Proposition 1.13(iii)]{LLS} (since $g$ must be a long root element in one of the simple factors of $M^0$), which gives $\dim X(g) = 162$ and $\a(g) = 3/4$. Finally, suppose $g$ is in the $A_1^2$ class, so $\dim g^G = 92$. To get $\a(g) < 2/3$ we need to show that 
\begin{equation}\label{e:eq3}
\dim (g^G \cap M) \leqs 19.
\end{equation} 

By inspecting \cite[Chapter 12]{Thomas}, we see that $V{\downarrow}M^0 = {\rm Lie}(A_2^4) \oplus W$, where $W$ is the module
\begin{align*}
& (\l_1 \otimes \l_1 \otimes 0 \otimes \l_2) \oplus (\l_1 \otimes \l_2 \otimes \l_1 \otimes 0) \oplus (\l_1 \otimes 0 \otimes \l_2 \otimes \l_1) \oplus (\l_2 \otimes \l_1 \otimes \l_2 \otimes 0) \\
& \oplus (\l_2 \otimes \l_2 \otimes 0 \otimes \l_1) \oplus (\l_2 \otimes 0 \otimes \l_1 \otimes \l_2) \oplus (0 \otimes \l_1 \otimes \l_1 \otimes \l_1) \oplus (0 \otimes \l_2 \otimes \l_2 \otimes \l_2).
\end{align*}
Here $\l_1$ denotes the natural $3$-dimensional module for $A_2$, $\l_2$ is its dual and $0$ is the trivial module. 

First assume $p=2$. If $y \in M$ is an involution with $\dim y^M > 19$, then $y$ must act as a graph automorphism on each $A_2$ factor. But from the structure of $W$ we deduce that $y$ has Jordan form $[J_2^{120},J_1^8]$ on $V$, so $y$ is in the class $A_1^4$. Now assume $p \ne 2$ and note that $g$ has Jordan form $[J_3^{14}, J_2^{64}, J_1^{78}]$ on $V$. If $g^G \cap (M \setminus M^0)$ is nonempty then $p=3$ and $g$ must cyclically permute three of the $A_2$ factors of $M^0$, possibly acting nontrivially on the fixed factor. But then $g$ has at least $54$ Jordan blocks of size $3$ on $V$, which is not compatible with elements in the $A_1^2$ class. Therefore $g^G \cap M \subseteq M^0$. To complete the argument, suppose there exists $y =y_1\cdots y_4 \in g^G \cap M^0$ with $\dim y^M > 19$. Then each $y_i$ is nontrivial and at least two are regular. If $p \geqs 5$ then the structure of $W$ implies that $y$ has Jordan blocks of size $4$ or more on $V$, which is a contradiction. On the other hand, if $p=3$ then $y$ has  Jordan form $[J_3^{72}]$ on $W$, which once again is incompatible with the Jordan form of $g$ on $V$. This justifies the bound in \eqref{e:eq3}.

Now suppose $M^0 = A_1^8$ and note that $M/M^0 = {\rm AGL}_{3}(2)$. If $g=u_{\a}$ then $\dim (g^G \cap M) = 2$ and we get $\dim X(g) = 168$, so $\a(g) = 3/4$. If $\dim g^G \geqs 112$ then the trivial bound $\dim (g^G \cap M) \leqs 24$ gives $\dim X(g) \leqs 136$ and thus $\a(g) \leqs 17/28$, so we have reduced to the case where $g$ is in the $A_1^2$ class. We claim that 
\[
\dim(g^G \cap M) \leqs 16,
\]
which is sufficient to show that $\a(g) < 2/3$. This is clear if $g^G \cap M \subseteq M^0$, so assume $g \in M \setminus M^0$, in which case $p \in \{2,3,7\}$. If $p \in \{2,3\}$ then it is easy to check that $\dim y^M \leqs 16$ for all $y \in M$ of order $p$, so let us assume $p=7$. If $y \in M \setminus M^0$ has order $p$ then from the decomposition of $V{\downarrow}M^0$ (see  \cite[Chapter 12]{Thomas}, for example) it is straightforward to show that the Jordan form of $y$ on $V$ will involve $J_7$ blocks, which is incompatible with the Jordan form of $g$ on $V$ (as noted above, $g$ has Jordan form $[J_3^{14}, J_2^{64}, J_1^{78}]$). This justifies the claim and completes the argument.

Finally, let us assume $M^0 = T_8$, so $M/M^0 = W(E_8) = 2.O_{8}^{+}(2)$. Clearly, $g^G \cap M = g^G \cap (M \setminus M^0)$ and $p$ must divide $|W(E_8)|$, so $p \in \{2,3,5,7\}$. If $g \ne u_{\a}$ then $\dim g^G \geqs 92$ and thus $\dim X(g) \leqs 156$, which gives $\a(g) \leqs 13/20$. On the other hand, if $g=u_{\a}$ then \cite[Proposition 1.13(iii)]{LLS} implies that $p=2$ and $\dim (g^G \cap M) = 1$, so $\dim X(g) = 183$ and $\a(g) = 61/80$, as recorded in Table \ref{tab:mr2}.
\end{proof}

\begin{lem}\label{l:mre7}
The conclusion to Proposition \ref{t:mr} holds if $G = E_7$.
\end{lem}

\begin{proof}
Here the cases to be considered are as follows:
\[
\begin{array}{cccccccc} \hline
M^0 & A_1D_6 & A_7 & A_2A_5 & A_1^3D_4 & A_1^7 & E_6T_1 &  T_7 \\ 
\dim X & 64 & 70 & 90 & 96 & 112 & 54 & 126  \\ \hline
\end{array} 
\]
Let $g \in \mathcal{P}$ be non-central and let $V = {\rm Lie}(G)$ be the Lie algebra of $G$. Also let $V_{56} = V_G(\l_7)$ be the $56$-dimensional minimal module for $G$.

First assume $g$ is semisimple of order $r$. If $M^0 \ne T_7, A_1^7$, then the lower bound on $\b(g)$ from \cite[Theorem 2(II)(b)]{LLS} yields $\a(g) < 2/3$. For $M^0 = T_7$ we have $\dim (g^G \cap M) \leqs 7$ and thus $\b(g) \geqs 47$ (since $\dim g^G \geqs 54$ for all nontrivial semisimple elements $g \in G$), hence $\dim X(g) \leqs 79$ and $\a(g) < 2/3$. 

To complete the analysis of semisimple elements, let us assume $M^0 = A_1^7$, so $\dim X = 112$ and $M/M^0 = {\rm GL}_{3}(2)$. If $C_G(g) \ne T_1E_6$ then  \cite[Theorem 2(II)(b)]{LLS} implies that $\b(g) \geqs 39$ and thus $\a(G) \leqs 73/112$. Now assume $C_G(g) = T_1E_6$, so $\dim g^G = 54$. We claim that $\dim(g^G \cap M) \leqs 16$, which yields $\a(g) \leqs 37/56$. To justify the claim, suppose we have $\dim (g^G \cap M) > 16$. Clearly, $\dim(g^G \cap M^0) \leqs 14$, so $g^G \cap (M \setminus M^0)$ must be nonempty and thus $r \in \{2,3,7\}$. If $r \in \{2,3\}$ then it is easy to check that $\dim(g^G \cap M) \leqs 14$, so $r=7$ is the only possibility and $g$ must cyclically permute the simple factors of $M^0$. By considering the restriction of $V$ to $M^0$ (see \cite[Chapter 12]{Thomas}), we deduce that $\dim C_V(g) \leqs \dim M^0 + 16 = 37$. But this implies that $\dim C_G(g) \leqs 37$ and we have therefore reached a contradiction. This justifies the claim and the proof of the lemma for semisimple elements is complete.

For the remainder we may assume $g$ is unipotent. If $M^0 = A_1D_6$ then $M$ is connected and the desired result is easily deduced from \cite{Law09}, where the $G$-class of each unipotent element in $M$ is determined. In particular, if $g=u_{\a}$ then $\dim X(g) = 48$ and $\a(g) = 3/4$. For all other unipotent elements, we get $\a(g) < 2/3$.

Next assume $M^0 = A_7$, so $M/M^0 = Z_2$. If $g=u_{\a}$ then \cite[Proposition 1.13(iii)]{LLS} implies that $g^G \cap M \subseteq M^0$ and we calculate that $\dim X(g) = 50$ and thus $\a(g) = 5/7$. For all other unipotent elements, we claim that $\a(g)<2/3$. If $g^G \cap M \subseteq M^0$ then Lawther's work in \cite{Law09} gives $\b(g) \geqs 28$, so $\dim X(g) \leqs 42$ and thus $\a(g) \leqs 3/5$. Finally, suppose $g^G \cap (M \setminus M^0)$ is nonempty, so $p=2$ and $g$ is an involution. As explained in the proof of \cite[Lemma 4.1]{LLS}, $g$ is in the class labelled $(A_1^3)^{(2)}$ or $A_1^4$. This implies that $\dim g^G - \dim (g^G \cap (M \setminus M^0)) \geqs 27$ and the claim follows. 

The case $M^0 = E_6T_1$ is very similar and we omit the details (as explained in the proof of \cite[Lemma 4.1]{LLS}, if $p=2$ and $g \in M \setminus M^0$, then $\dim g^G - \dim(g^G \cap (M \setminus M^0)) \geqs 27$). 

A similar argument applies when $M^0= A_2A_5$. For example, suppose $p=2$ and $g \in M \setminus M^0$. Here $\dim(g^G \cap (M \setminus M^0))=19$ or $25$, and $\dim g^G \geqs 52$ (since $g \ne u_{\a}$ by \cite[Proposition 1.13(iii)]{LLS}). Moreover, if $\dim(g^G \cap (M \setminus M^0))=25$ then by considering $V{\downarrow}M^0$ we deduce that $g$ has Jordan form $[J_2^{63},J_1^7]$ on $V$, which means that $g$ is in the $A_1^4$ class. Therefore, $\dim g^G - \dim(g^G \cap (M \setminus M^0)) \geqs 33$ and we conclude that $\a(g) \leqs 19/30$.

To complete the proof, we may assume that 
\[
M^0 \in \{A_1^3D_4, A_1^7, T_7\}.
\]

Suppose $M^0 = A_1^3D_4$, so $M/M^0 = S_3$. First assume $g=u_{\a}$, so $\dim g^G = 34$ and $\dim (g^G \cap M^0) = 10$ (the class of long root elements in the $D_4$ factor is $10$-dimensional). If $g^G \cap (M \setminus M^0)$ is nonempty, then $p=2$ and \cite[Proposition 1.13(iii)]{LLS} implies that $\dim (g^G \cap (M \setminus M^0)) = 7$. Therefore, $\b(g)=24$ and thus $\a(g) = 3/4$.

If in fact $\dim g^G \geqs 64$, then the trivial bound $\dim (g^G \cap M) \leqs 30$ implies that $\dim X(g) \leqs 62$ and $\a(g) < 2/3$. Thus we may assume $g$ is contained in one of the classes labelled $A_1^2$ and $(A_1^3)^{(2)}$. If $p \ne 2$ then the proof of \cite[Proposition 5.12]{BGS} gives $\dim (g^G \cap M) \leqs 14$ and we conclude that $\dim X(g) \leqs 58$. Now assume $p=2$. We will consider the two possibilities for $g^G$ in turn.

First assume $g$ is in the class $A_1^2$, so $\dim g^G = 52$ and $g$ has Jordan form $[J_2^{20},J_1^{16}]$ on $V_{56}$ (see \cite[Table 7]{Lawunip}). We claim that $\dim(g^G \cap M) \leqs 19$, which yields $\a(g) \leqs 21/32$. If $y = y_1 \cdots y_4 \in g^G \cap M^0$ then $\dim y^M > 19$ only if $y_4 \in D_4$ is a $c_4$-type involution (in the notation of \cite{AS}). Then by considering the decomposition of $V_{56}{\downarrow}M^0$ (see \cite[Chapter 12]{Thomas}), we deduce that $y$ has at least $24$ Jordan blocks of size $2$ on $V_{56}$, which is a contradiction. Similarly, if $y \in g^G \cap (M \setminus M^0)$ and $\dim y^M > 19$ then we calculate that $y$ has Jordan form $[J_2^{28}]$ on $V_{56}$ and once again we have reached a contradiction. 

Now assume $g$ is in the class $(A_1^3)^{(2)}$, so $\dim g^G = 54$ and it suffices to show that $\dim(g^G \cap M) \leqs 21$. If $y \in g^G \cap (M \setminus M^0)$ then $y$ acts as a graph automorphism on the $D_4$ factor of $M^0$ and we deduce that $\dim y^M \leqs 21$. Similarly, if $y \in g^G \cap M^0$ and $\dim y^M > 21$ then $y = y_1 \cdots y_4$, where each $y_i$ is an involution and $y_4 \in D_4$ is of type $c_4$. From the decomposition of $V{\downarrow}M^0$, we calculate that $y$ has Jordan form $[J_2^{63},J_1^7]$ on $V$, which is incompatible with the action of $g$ on $V$ (see \cite[Table 8]{Lawunip}).

Next assume $M^0 = A_1^7$, in which case $M/M^0 = {\rm GL}_{3}(2)$. If $g=u_{\a}$ then $\dim(g^G \cap M) = 2$, giving $\dim X(g) = 80$ and $\a(g) = 5/7$. If $\dim g^G \geqs 64$ then $\dim X(g) \leqs 69$ and thus $\a(g) <2/3$, so we may assume that $g$ is in the $A_1^2$ or $(A_1^3)^{(2)}$ class. Again we claim that $\a(g) <2/3$. Here $\dim g^G \geqs 52$ and $\dim (g^G \cap M^0) \leqs 14$, so we may assume $g \in M \setminus M^0$ and thus $p \in \{2,3,7\}$. If $p \in \{2,3\}$ then it is easy to check that $\dim (g^G \cap (M \setminus M^0)) \leqs 14$ and the claim follows. Finally, suppose $p=7$ and $g$ cyclically permutes the $A_1$ factors of $M^0$. By considering the restriction of $V$ to $M^0$ (see \cite[Chapter 12]{Thomas}), we deduce that $g$ has $J_7$ blocks in its Jordan form on $V$, but this is not compatible with elements in the $A_1^2$ and $(A_1^3)^{(2)}$ classes (see \cite[Table 8]{Lawunip}). This completes the proof of the claim.

Finally, let us assume $M^0=T_7$, so $p$ divides $|M/M^0| = |W(E_7)|$ and thus $p \in \{2,3,5,7\}$. If $g \ne u_{\a}$ then $\dim g^G \geqs 52$ and thus $\dim X(g) \leqs 81$, which gives $\a(g) \leqs 9/14$. On the other hand, if $g = u_{\a}$ is a long root element then \cite[Proposition 1.13(iii)]{LLS} implies that $p=2$ and $\dim (g^G \cap M) = 1$, so $\dim X(g) = 93$ and $\a(g) = 31/42$.
\end{proof}

\begin{lem}\label{l:mre6}
The conclusion to Proposition \ref{t:mr} holds if $G = E_6$.
\end{lem}

\begin{proof}
Here $M^0$ is one of the following:
\[
\begin{array}{ccccc} \hline
M^0 & A_1A_5 & A_2^3 & D_4T_2 & T_6 \\ 
\dim X & 40 & 54 & 48 & 72   \\ \hline
\end{array} 
\]
Let $V$ be the Lie algebra of $G$ and let $V_{27} = V_G(\l_1)$ be the $27$-dimensional minimal module for $G$. As before, if $g \in \mathcal{P}$ is semisimple then the result follows from the lower bound on $\b(g)$ given in \cite[Theorem 2(II)(b)]{LLS}. For the remainder, let us assume $g$ is unipotent. 

The case $M^0 = A_1A_5$ is very straightforward. Here $M$ is connected and the result follows from the fusion computations in \cite{Law09}. 

Next assume $M^0 = A_2^3$, so $M/M^0 = S_3$. If $g = u_{\a}$ then $\dim g^G = 22$ and by applying \cite[Proposition 1.13]{LLS} we deduce that $\dim (g^G \cap M) = 4$, so $\dim X(g) = 36$ and $\a(g) = 2/3$. In every other case, we have  $\dim g^G \geqs 32$ and \cite{Law09} gives $\dim g^G - \dim (g^G \cap M^0) \geqs 24$, so we may assume $g \in M \setminus M^0$, in which case $p \in \{2,3\}$. For $p=2$, one checks that $\dim (g^G \cap (M \setminus M^0)) \leqs 12$, so $\b(g) \geqs 20$ and this yields $\a(g) < 2/3$. For $p=3$, we have $\dim (g^G \cap (M \setminus M^0))=16$ and by considering $V_{27}{\downarrow}M^0$ (see \cite[Chapter 12]{Thomas}) we calculate that $g$ has Jordan form $[J_3^9]$ on $V_{27}$. Then according to \cite[Table 5]{Lawunip}, $g$ is in one of the classes labelled $A_2^2$ or $A_2^2A_1$, so $\dim g^G \geqs 48$ and thus $\b(g) \geqs 24$. The result follows.

Next assume $M^0 = D_4T_2$, so $M/M^0=S_3$. If $g=u_{\a}$ then $\dim (g^G \cap M^0) = 10$ and $\dim (g^G \cap (M \setminus M^0)) \leqs 7$, so $\a(g) = 3/4$. Now assume $\dim g^G \geqs 32$. As explained in the proof of \cite[Proposition 5.13]{BGS}, we have $\dim (g^G \cap M) < \frac{1}{2}\dim g^G$, so $\b(g) \geqs 17$ and $\a(g) \leqs 31/48$.

To complete the proof of the lemma, we may assume that $M^0 = T_6$. If $g \ne u_{\a}$ then $\dim g^G \geqs 32$, so $\dim X(g) \leqs 46$ and $\a(g) < 2/3$. On the other hand, if $g=u_{\a}$, then $p=2$, $\dim g^G = 22$ and $\dim(g^G \cap M) =1$, which gives $\dim X(g) = 51$ and $\a(g) = 17/24$.
\end{proof}

\begin{lem}\label{l:mrf4}
The conclusion to Proposition \ref{t:mr} holds if $G = F_4$.
\end{lem}

\begin{proof}
The cases here are as follows:
\[
\begin{array}{ccccccc} \hline
M^0 & B_4 & C_4 \, (p=2) & D_4 & \tilde{D}_4 \, (p=2) & A_1C_3 \, (p \ne 2) & A_2\tilde{A}_2 \\
\dim X & 16 & 16 & 24 & 24 & 28 & 36   \\ \hline
\end{array} 
\]
Let $V$ be the Lie algebra of $G$ and write $V_{26}$ for the $26$-dimensional Weyl module $W_G(\l_4)$. Let $g$ be an element of $\mathcal{P}$.

First assume $g$ is semisimple. By applying \cite[Theorem 2(II)(b)]{LLS}, we can immediately reduce to the cases $M^0 \in \{A_1C_3, A_2\tilde{A}_2\}$ with $p \ne 2$ and $g$ a $B_4$-involution, so $\dim g^G = 16$. Suppose $M^0 = A_1C_3$ and note that $M$ is connected. Here \cite[Theorem 2(II)(b)]{LLS} yields $\b(g) \geqs 8$, so $\dim X(g) \leqs 20$ and we claim that equality holds. To see this, we can use the restriction  
\[
V{\downarrow}M^0 = {\rm Lie}(A_1C_3) \oplus (V_{A_1}(\l_1) \otimes V_{C_3}(\l_3))
\]
to compute the eigenvalues on $V$ of each involution in $M^0$. For example, if we take $y=y_1y_2 \in M^0$, where $y_1=1$ and $y_2 = [-I_4, I_2] \in C_3$, then $y$ has Jordan form $[-I_{16},I_{36}]$ on $V$, so $\dim C_G(y)= \dim C_V(y) = 36$ and thus $y$ is a $B_4$-involution. Since $\dim y^{M} = 8$, the claim follows and we conclude that $\a(g) = 5/7$. 

The case $M^0 = A_2\tilde{A}_2$ is similar. Here \cite[Theorem 2(II)(b)]{LLS} gives $\b(g) \geqs 12$, so $\dim X(g) \leqs 24$ and we claim that equality holds, so $\a(g) = 2/3$. It suffices to show that the involution $y = y_1y_2 \in M^0$, where $y_1=1$ and $y_2 = [-I_2, I_1]$, is of type $B_4$. This is an easy calculation, either by consulting the decomposition of $V{\downarrow}M^0$ given in \cite[Chapter 12]{Thomas}, or by arguing as follows. Fix a maximal torus $T$ and simple roots $\a_1, \a_2, \a_3,\a_4$ for $G$. Let $\a_0$ be the highest root and let $U_{\b}$ be the root subgroup of $G$ corresponding to the root $\b$. We may assume $M^0$ has simple roots $-\a_0,\a_1,\a_3,\a_4$ and in terms of the standard Lie notation we may take $y = h_{\a_3}(-1)$. Then 
\[
C_G(y) = \la T, U_{\b}\,:\, \b = \sum_{i=1}^{4}n_i\a_i \mbox{ with $n_4$ even} \ra = B_4
\]
with simple roots $-\a_0, \a_1,\a_2,\a_3$.

For the remainder, we may assume that $g$ is unipotent. The cases $M^0 \in \{A_1C_3, B_4, C_4\}$ are straightforward to handle, using the fusion computations in \cite{Law09} (note that in each case, $M=M^0$). 

Next we turn to the case $M^0=D_4$, so $M/M^0 = S_3$ and $D_4$ is generated by long root subgroups. If $g=u_{\a}$ is a long root element, then $\dim g^G = 16$ and $\dim (g^G \cap M) = 10$, which gives $\a(g) = 3/4$. For the remainder, let us assume $g \ne u_{\a}$. If $y \in g^G \cap M^0$ then we can compute the Jordan form of $y$ on $V_{26}$ by considering the restriction
\[
V_{26}{\downarrow}M^0 = V_{D_4}(\l_1) \oplus V_{D_4}(\l_3) \oplus V_{D_4}(\l_4) \oplus 0^2.
\]
In this way, one checks that $\dim g^G - \dim (g^G \cap M^0) \geqs 10$. Now assume 
$g^G \cap (M \setminus M^0)$ is nonempty, so $p \in \{2,3\}$. We claim that 
\begin{equation}\label{e:eq4}
\dim g^G - \dim (g^G \cap (M \setminus M^0)) \geqs 9.
\end{equation}
For $p=2$, there are two classes of involutions in $M \setminus M^0$ (with representatives $b_1$ and $b_3$ in the notation of \cite{AS}) and we have $\dim (g^G \cap (M \setminus M^0)) = 7$ or $15$. Since $\dim g^G \geqs 16$, we may assume $g \in M$ is a $b_3$-involution. From the above decomposition of $V_{26}{\downarrow}M^0$, we calculate that $g$ has Jordan form $[J_2^{12},J_1^2]$ on $V_{26}$, so \cite[Table 3]{Lawunip} implies that $g$ is in the class $A_1\tilde{A}_1$. In particular, $\dim g^G = 28$ and the claim follows. For $p=3$ we can verify the bound in \eqref{e:eq4} by inspecting the proof of \cite[Proposition 5.14]{BGS}. Therefore, we conclude that if $g \ne u_{\a}$ then $\dim g^G - \dim(g^G \cap M) \geqs 9$ and thus 
$\a(g) \leqs 5/8$.

The case $M^0 = \tilde{D}_4$ (with $p=2$) is entirely similar. Here $M^0$ is generated by short root subgroups, and we note that the subgroups $D_4$ and $\tilde{D_4}$ are interchanged by a graph automorphism of $G$. Therefore, $\a(g)=3/4$ if $g \in G$ is a short root element, otherwise $\a(g) \leqs 5/8$.

Finally, let us assume $M^0 = A_2\tilde{A}_2$, so $M/M^0 = Z_2$. If $g=u_{\a}$ then $\dim g^G = 16$ and $\dim (g^G \cap M) = 4$, which gives $\a(g) = 2/3$. The same conclusion holds if $p=2$ and $g$ is a short root element. Now assume $g \ne u_{\a},u_{\b}$. From the fusion computations in \cite{Law09} we deduce that $\dim g^G - \dim (g^G \cap M^0) \geqs 18$. Now assume $g^G \cap (M \setminus M^0)$ is nonempty, so $p=2$ and $\dim (g^G \cap (M \setminus M^0)) = 10$ (since $g$ acts as a graph automorphism on both factors of $M^0$). By considering the restriction of $V_{26}$ to $M^0$ (see \cite[Chapter 12]{Thomas}), we calculate that $g$ has Jordan form $[J_2^{12},J_1^2]$ on $V_{26}$, so \cite[Table 3]{Lawunip} implies that $g$ is in the class $A_1\tilde{A}_1$ and thus $\dim g^G =  28$. We conclude that if $g \ne u_{\a},u_{\b}$ then 
$\b(g) \geqs 18$ and this gives $\a(g) \leqs 1/2$.  
\end{proof}

\begin{lem}\label{l:mrg2}
The conclusion to Proposition \ref{t:mr} holds if $G = G_2$.
\end{lem}

\begin{proof}
Here the possibilities for $M^0$ are as follows:
\[
\begin{array}{cccc} \hline
M^0 & A_1\tilde{A}_1 & A_2 & \tilde{A}_2 \, (p=3) \\ 
\dim X & 8 & 6 & 6     \\ \hline
\end{array} 
\]
Let $V_7$ be the $7$-dimensional Weyl module $W_G(\l_1)$ and fix an element $g \in \mathcal{P}$.

If $g$ is semisimple, then the lower bound on $\b(g)$ in \cite[Theorem 2(II)(b)]{LLS} implies that $\a(g) \leqs 1/2$, so we may assume that $g$ is unipotent. For $M = M^0 = A_1\tilde{A}_1$, the fusion computations in \cite{Law09} imply that $\a(g) \leqs 1/2$. 

Next suppose $M^0 = A_2$, so $M/M^0 = Z_2$ and $M^0$ is generated by long root subgroups. If $g=u_{\a}$ then $\dim g^G = 6$ and $\dim (g^G \cap M) = 4$, so $\a(g) = 2/3$. Now assume $g \ne u_{\a}$. If $g^G \cap (M \setminus M^0)$ is nonempty, then $p=2$ and $\dim (g^G \cap (M \setminus M^0)) = 5$. Moreover, by considering the restriction of $V_7$ to $M^0$, we deduce that $g$ has Jordan form $[J_2^3,J_1]$ on $V_7$, so $g$ is in the class $\tilde{A}_1$ (see \cite[Table 1]{Lawunip}) and $\dim g^G=8$. Similarly, if $g^G \cap M^0$ is nonempty then $g$ is in the class labelled $G_2(a_1)$ (see \cite{Law09}), so $\dim(g^G \cap M^0) = 6$ and $\dim g^G = 10$. We conclude that $\a(g) \leqs 1/2$ if $g \ne u_{\a}$.

Finally, the case $M^0 = \tilde{A}_2$ (with $p=3$) is entirely similar: if $g$ is a short root element, then $\a(g) = 2/3$, otherwise $\a(g) \leqs 1/2$.
\end{proof}

\vs

This completes the proof of Proposition \ref{t:mr}.
 
\subsection{Remaining subgroups}\label{s:rem}

To complete the proof of Theorem \ref{t:mainex}, it remains to handle the cases arising in Table \ref{tab:nmr} (together with the two special cases recorded in parts (iii) and (iv) of Theorem \ref{t:ls}).

\begin{prop}\label{t:nmr}
Let $G$ be a simply connected simple algebraic group of exceptional type, let $M$ be a maximal positive dimensional subgroup of $G$ with ${\rm rank}\,M^0 < {\rm rank}\,G$ and let $g \in \mathcal{P}$ be non-central. Then 
$\a(G,M,g) \geqs \frac{2}{3}$ if and only if $(G,M,g)$ is one of the cases recorded in Table \ref{tab:nmr2}.
\end{prop}

\renewcommand{\arraystretch}{1.1}
\begin{table}[h]
\[
\begin{array}{llcc} \hline
G & M^0 & g &  \a(G,M,g)  \\ \hline
E_8 & G_2F_4 & u_{\a} & 10/13  \\
& A_1G_2^2\, (p \ne 2) & u_{\a} & 165/217  \\
E_7 &  A_1F_4 & u_{\a} & 10/13 \\
& G_2C_3 &  u_{\a} & 5/7  \\
E_6 &  F_4 & u_{\a} & 10/13 \\
& A_2G_2 & u_{\a} & 5/7 \\
&  C_4\, (p \ne 2) & u_{\a} & 2/3 \\
F_4 & A_1G_2 & u_{\a} & 5/7 \\ \hline
\end{array}
\]
\caption{$M^0$ not maximal rank, $\a(G,M,g) \geqs 2/3$}
\label{tab:nmr2}
\end{table}
\renewcommand{\arraystretch}{1}

\begin{lem}\label{l:nmre8}
The conclusion to Proposition \ref{t:nmr} holds if $G = E_8$.
\end{lem}

\begin{proof}
We start by considering the special case recorded in part (iv) of Theorem \ref{t:ls}, so $p \ne 2,3,5$ and $M = A_1 \times S_5$. Let $g \in M$ be an element in $\mathcal{P}$. From the construction of $M^0$ as a diagonal subgroup of $A_1A_1 < A_4A_4$, it is clear that $g \ne u_{\a}$. Therefore, $\dim g^G \geqs 92$ and thus $\dim X(g) \leqs 155$ since $\dim (g^G \cap M) \leqs 2$. This yields $\a(g) \leqs 31/49$.

To complete the proof of the lemma, we need to handle the following cases:
\[
\begin{array}{cccccc} \hline
M^0 & A_1\, (p \geqs 23; 29; 31) & B_2 \, (p \geqs 5) & A_1A_2\, (p \ne 2,3) & A_1G_2^2 \, (p \ne 2) & G_2F_4 \\ 
|M/M^0| & 1 & 1 & 2 & 2 & 1 \\
\dim X & 245 & 238 & 237 & 217 & 182 \\ \hline
\end{array}
\]
Let $g \in M$ be an element in $\mathcal{P}$ and let $V$ be the Lie algebra of $G$. 

The cases $M^0 \in \{A_1, B_2, A_1A_2\}$ are very straightforward. For instance, let us assume $M^0=A_1A_2$, so $p \ne 2,3$ and $M/M^0=Z_2$. By \cite{Law09}, $g \ne u_{\a}$ and thus $\dim g^G \geqs 92$ and $\dim (g^G \cap M) \leqs 8$, so $\dim X(g) \leqs 153$ and $\a(g)<2/3$.

Next assume $M^0 = A_1G_2^2$, so $p \ne 2$ and $M/M^0=Z_2$. If $g$ is not a unipotent element in the class $A_1$ or $A_1^2$, then $\dim g^G \geqs 112$ and $\dim (g^G \cap M) \leqs 26$, so $\dim X(g) \leqs 131$ and we get $\a(g)<2/3$. If $g$ is in the $A_1$ class then $\dim g^G = 58$ and \cite[Proposition 1.13(ii)]{LLS} implies that $\dim (g^G \cap M) =6$, so $\dim X(g) = 165$ and thus $\a(g) = 165/217$. Finally, suppose $g$ is in the class labelled $A_1^2$, so $\dim g^G = 92$ and $g$ has Jordan form $[J_3^{14}, J_2^{64},J_{1}^{78}]$ on $V$. Here it will be useful to consider the restriction of $V$ to $M^0$. According to \cite[Chapter 12]{Thomas}, we have $V{\downarrow}M^0 = {\rm Lie}(A_1G_2^2) \oplus W$, where
\begin{equation}\label{e:eq6}
W = (W_{A_1}(2\l_1) \otimes U \otimes U) \oplus (W_{A_1}(4\l_1) \otimes U \otimes 0) \oplus (W_{A_1}(4\l_1) \otimes 0 \otimes U)
\end{equation}
and $U = W_{G_2}(\l_1)$. We claim that $\dim (g^G \cap M) \leqs 18$, so $\dim X(g) \leqs 143$ and $\a(g)<2/3$. 

To justify the claim, suppose $y=y_1y_2y_3 \in M^0$ is $G$-conjugate to $g$. If $y_1 \ne 1$ and $p \ne 3$ then $y_1$ has Jordan form $J_5$ on $W_{A_1}(4\l_1)$ and we deduce from \eqref{e:eq6} that the Jordan form of $y$ on $V$ is incompatible with the form of $g$. Similarly, if $y_1 \ne 1$ and $p=3$ then $y_1$ has Jordan form $J_3$ on $W_{A_1}(2\l_1)$ and thus $y$ acts as $[J_3^{49}]$ on the first summand in \eqref{e:eq6}. Again, this is a contradiction and thus $y_1=1$. If $y_2$ or $y_3$ is regular, then $p \ne 3,5$ (since $y \in \mathcal{P}$) and $y_i$ acts as $[J_7]$ on $U$. Again, we see that this is incompatible with the Jordan form of $g$, so neither $y_2$ nor $y_3$ is regular. Similarly, $y_2$ and $y_3$ are not both in the $G_2(a_1)$ class (indeed, such an element has Jordan form $[J_3^2,J_1]$ on $U$ and thus $y$ would have more than $14$ Jordan blocks of size $3$ on $V$). Therefore, $\dim y^M \leqs 18$ and this establishes the claim.

Finally, let us assume $M^0 = G_2F_4$, so $M$ is connected. First assume $g$ is unipotent and note that the fusion of unipotent classes is determined in \cite{Law09}. If $g=u_{\a}$ then $\dim g^G = 58$ and $\dim (g^G \cap M) = 16$, so $\dim X(g) = 140$ and $\a(g) = 10/13$. In every other case, one checks that $\b(g) \geqs 70$, so $\dim X(g) \leqs 112$ and the result follows. Now assume $g$ is semisimple. If $\dim g^G \geqs  128$ then the trivial bound $\dim (g^G \cap M) \leqs 60$ yields $\a(g) \leqs 57/91$. Therefore, to complete the argument we may assume that $C_G(g) = E_7A_1$ or $E_7T_1$. In the former case, $\dim g^G = 112$, $p \ne 2$ and $g$ is an involution, so $\dim (g^G \cap M) \leqs 8+28=36$ and thus $\dim X(g) \leqs 106$. Finally, suppose $C_G(g) = E_7T_1$. Here $\dim g^G = 114$ and we claim that $\dim (g^G \cap M) \leqs 52$, which gives $\dim X(g) \leqs 120$ and $\a(g) \leqs 60/91$. 

To see this, first observe that $V {\downarrow}M = {\rm Lie}(G_2F_4) \oplus W$, where $W=W_{G_2}(\l_1) \otimes W_{F_4}(\l_4)$. Seeking a contradiction, suppose there exists an element $y=y_1y_2 \in g^G \cap M$ with $\dim y^M > 52$ and note that both $y_1$ and $y_2$ are nontrivial. Let 
\begin{equation}\label{e:nudef}
\nu(y,W) = \min\{\dim \, [W, \l y] \, :\, \l \in k^*\}
\end{equation}
be the codimension of the largest eigenspace of $y$ on $W$ and set $s=\nu(y,W)$. Similarly, define $s_1 = \nu(y_1,V_7)$ with respect to the action of $y_1$ on $V_7 = W_{G_2}(\l_1)$. Since $y_1 \ne 1$, it is easy to check that $s_1 \geqs 3$ and thus $s \geqs 3.26 = 78$. This implies that the dimension of the $1$-eigenspace of $y$ on $W$ is at most $104$. Moreover, since $\dim y^M > 52$ it follows that $\dim C_{M}(y) \leqs 12$ and thus the $1$-eigenspace of $y$ on $V$ is at most $116$-dimensional. But this implies that $\dim C_G(y) \leqs 116$ and we have reached a contradiction. This justifies the claim and the proof of the lemma is complete.
\end{proof}

\begin{lem}\label{l:nmre7}
The conclusion to Proposition \ref{t:nmr} holds if $G = E_7$.
\end{lem}

\begin{proof}
Let $g \in M$ be an element in $\mathcal{P}$ and let $V$ be the Lie algebra of $G$. 

First suppose $p \ne 2$ and $M = (2^2 \times D_4).S_3$, as in part (iii) of Theorem \ref{t:ls}. Note that $\dim X = 105$. Suppose $g \in M$ is semisimple of prime order $r$. If $\dim g^G \geqs 64$ then the trivial bound $\dim (g^G \cap M) \leqs 24$ yields $\dim X(g) \leqs 65$, which gives $\a(g)<2/3$. Now assume $\dim g^G < 64$, so $C_G(g) = E_6T_1$ and $\dim g^G = 54$. We claim that $\dim (g^G \cap M) \leqs 18$, which gives $\a(g)<2/3$. This is clear if $r=2$, so we may assume $r>2$. Let $y \in M$ be an element of order $r$ with $\dim y^M > 18$ and let $s$ be the dimension of the $1$-eigenspace of $y$ on $V$. To establish the claim, we need to show that $s<79$. To do this, it will be helpful to consider the decomposition
\begin{equation}\label{e:eq7}
V{\downarrow}M^0 = {\rm Lie}(D_4) \oplus V_{D_4}(2\l_1) \oplus V_{D_4}(2\l_3) \oplus V_{D_4}(2\l_4)
\end{equation}
(see \cite[Chapter 12]{Thomas}). Note that the $35$-dimensional module $V_{D_4}(2\l_1)$ is the nontrivial composition factor of the symmetric-square $S^2V_8$, where $V_8$ is the natural module for $M^0$. In addition, $V_{D_4}(2\l_3)$ and $V_{D_4}(2\l_4)$ are the images of this module under a triality graph automorphism of $M^0$. 

If $y \in M \setminus M^0$ then $r=3$ and $y$ is a triality graph automorphism; since $y$ cyclically permutes the $V_{D_4}(2\l_i)$ summands, it follows that $s=35+t$, where $t$ is the dimension of the $1$-eigenspace of $y$ on ${\rm Lie}(D_4)$. In particular, $s<79$ as required. Now assume $y \in M^0$ and note that the condition $\dim y^M > 18$ implies that $r \geqs 5$. The $1$-eigenspace of $y$ on ${\rm Lie}(D_4)$ is at most $8$-dimensional. Furthermore, one can check that the $1$-eigenspace of $y$ on each $V_{D_4}(2\l_i)$ summand is at most $11$-dimensional. For example, if $r=5$ and $y = [I_4,w,w^2,w^3,w^4]$ with respect to the natural module (where $w \in k$ is a primitive  fifth root of unity), then $\dim y^M = 20$ and $y$ has an $11$-dimensional $1$-eigenspace on $V_{D_4}(2\l_1)$. In particular, we deduce that $s \leqs 41$ and this completes the proof of the claim.

Now suppose $g$ is unipotent. As above, we may assume that 
$\dim g^G<64$, so $g$ belongs to one of the classes labelled $A_1$, $A_1^2$ and $(A_1^3)^{(2)}$. We can rule out long root elements by repeating the argument in the proof of \cite[Proposition 5.12]{BGS}, so we just need to consider the classes $A_1^2$ and $(A_1^3)^{(2)}$. If $g^G \cap (M \setminus M^0)$ is nonempty then $p=3$ and $g$ acts as a triality graph automorphism on $M^0$. Moreover, the decomposition in \eqref{e:eq7} implies that $g$ has at least $35$ Jordan blocks of size $3$ on $V$ and by inspecting \cite[Table 8]{Lawunip} we conclude that $g$ is not in $A_1^2$ nor in $(A_1^3)^{(2)}$. Therefore, $g^G \cap M \subseteq M^0$. Suppose $y \in M^0$ has order $p$. If $y$ is a long root element in $M^0$ then using \eqref{e:eq7} we calculate that $y$ has Jordan form $[J_3^{10},J_2^{32},J_1^{39}]$ on $V$ and thus $y$ is in the $A_1^2$ class of $G$ (see \cite[Table 8]{Lawunip}). Moreover, if $y$ is any other unipotent element in $M^0$ then the Jordan form of $y$ on $V$ is incompatible with the Jordan form of elements in the $A_1^2$ and $(A_1^3)^{(2)}$ classes. Therefore, 
$\dim (g^G \cap M) = 10$ if $g$ is in $A_1^2$ (and this gives $\a(g) = 3/5$) and we conclude that $M$ does not contain any elements in the $(A_1^3)^{(2)}$ class. 

For the remainder, we may assume $M^0$ is one of the following:
\[
\begin{array}{ccccccc} \hline
M^0 & A_1\, (p \geqs 17; 19) & A_2 \, (p \geqs 5) & A_1^2\, (p \ne 2,3) & A_1G_2 \, (p \ne 2) & A_1F_4 & G_2C_3 \\ 
|M/M^0| & 1 & 2 & 1 & 1 & 1 & 1 \\
\dim X & 130 & 125 & 127 & 116 & 78 & 98 \\ \hline
\end{array}
\]

The cases $M^0 \in \{A_1, A_2, A_1^2, A_1G_2\}$ are straightforward. For example, suppose $M^0 = A_1G_2$, so $p \ne 2$ and $\dim X = 116$. If $g$ is semisimple  then $\dim g^G \geqs 54$ and $\dim (g^G \cap M) \leqs 14$, so $\dim X(g) \leqs 76$ and $\a(g) \leqs 19/29$. Now assume $g$ is unipotent. If $\dim g^G \geqs 54$ then once again we see that $\dim X(g) \leqs 76$, so we can assume $g$ is in the class $A_1$ or $A_1^2$. The fusion of unipotent classes in $M$ is determined in \cite{Law09} and we see that $g \ne u_{\a}$. In addition, if $g$ is in the $A_1^2$ class then $\dim g^G = 52$ and $\dim (g^G \cap M) = 6$, so $\dim X(g) = 70$.

Next assume that $M^0 = A_1F_4$, so $M$ is connected and $\dim X = 78$. If $g=u_{\a}$ then $\dim g^G = 34$ and $\dim(g^G \cap M) = 16$ (see \cite{Law09}), so $\dim X(g)=60$ and thus $\a(g) = 10/13$. In every other case, $\dim g^G \geqs 52$ and the proof of \cite[Proposition 5.12]{BGS} gives $\dim (g^G \cap M) < \frac{1}{2}\dim g^G$. Therefore 
$\b(g) \geqs 28$ and thus $\dim X(g) \leqs 50$. The result follows.

Finally, suppose $M^0=G_2C_3$, so $M$ is connected and $\dim X = 98$. First assume $g$ is unipotent and note that the fusion of unipotent classes is determined in \cite{Law09}. If $g=u_{\a}$ then $\dim g^G = 34$ and $\dim (g^G \cap M) = 6$, giving $\dim X(g) = 70$ and $\a(g) =5/7$. In every other case, one checks that $\b(g) \geqs 40$, so $\dim X(g) \leqs 58$ and $\a(g) < 2/3$. Now assume $g$ is semisimple and note that $\dim (g^G \cap M) \leqs 30$. In particular, if $\dim g^G \geqs 64$ then $\dim X(g) \leqs 64$ and $\a(g) < 2/3$. Therefore, we may assume $C_G(g) = E_6T_1$, so $\dim g^G = 54$. We claim that $\dim (g^G \cap M) \leqs 20$, which gives $\a(g) < 2/3$. 

If $p \ne 2$ and $g$ is an involution, then $\dim (g^G \cap M) \leqs 8+12 = 20$ and the claim follows. Now assume $g$ has odd order and consider the restriction
$V{\downarrow}M = {\rm Lie}(G_2C_3) \oplus W$, where $W = W_{G_2}(\l_1) \otimes W_{C_3}(\l_2)$ (see \cite[Chapter 12]{Thomas}). Suppose there exists $y=y_1y_2 \in g^G \cap M$ with $\dim y^M > 20$. Note that both $y_1$ and $y_2$ are nontrivial. By arguing as in the proof of the previous lemma, we calculate that $\nu(y,W) \geqs 3.14 = 42$ (see \eqref{e:nudef}) and thus the dimension of the $1$-eigenspace of $y$ on $W$ is at most $56$. The bound $\dim y^M>20$ implies that $\dim C_{M}(y) \leqs 13$ and we conclude that  $\dim C_G(y) \leqs 69$, which is a contradiction. This justifies the claim and completes the argument.
\end{proof}

\begin{lem}\label{l:nmre6}
The conclusion to Proposition \ref{t:nmr} holds if $G = E_6$.
\end{lem}

\begin{proof}
We need to consider the following possibilities for $M^0$:
\[
\begin{array}{cccccc} \hline
M^0 & A_2 \, (p \ne 2,3) & G_2 \, (p \ne 7) & C_4 \, (p \ne 2) & F_4 & A_2G_2 \\
|M/M^0| & 2 & 1 & 1 & 1 & 1 \\
\dim X & 70 & 64 & 42 & 26 & 56 \\ \hline
\end{array}
\]
In particular, note that $M$ is connected in the final case (as explained in \cite[Remark 5.2(ii)]{BT}, the value $t=2$ given in \cite[Table 10.1]{LS04} should be $t=1$). Let $g \in M$ be an element in $\mathcal{P}$ and let $V$ be the Lie algebra of $G$.

First assume $M^0 = A_2$, so $p \ne 2,3$, $\dim X = 70$ and $M/M^0=Z_2$. By \cite{Law09}, $M$ does not contain any long root elements of $G$, so $\dim g^G \geqs 32$ and $\dim(g^G \cap M) \leqs 6$, which gives $\dim X(g) \leqs 44$ and $\a(g) \leqs 22/35$.

Next suppose $M^0 = G_2$, so $p \ne 7$, $M$ is connected and $\dim X = 64$. If $g$ is unipotent, then \cite{Law09} implies that $\dim g^G \geqs 40$ and the trivial bound 
$\dim (g^G \cap M) \leqs 12$ yields $\a(g) < 2/3$. Now assume $g$ is semisimple. By the previous argument, the result follows if $\dim g^G \geqs 40$, so we may assume $C_G(g) = D_5T_1$ and $\dim g^G = 32$. We claim that $\dim (g^G \cap M) \leqs 10$, which gives $\a(g) \leqs 21/32$. This is clear if $g$ is an involution, so let us assume $g$ has prime order $r>2$. Now $V{\downarrow}M = {\rm Lie}(G_2) \oplus V_{64}$,
where $V_{64} = W_{G_2}(\l_1+\l_2)$ (see \cite[Chapter 12]{Thomas}). By considering the set of weights of a maximal torus of $G$ on $V_{64}$, it is straightforward to check that the dimension of the $1$-eigenspace of $g$ on $V_{64}$ is at most $20$ (with equality only if $r=3$). Therefore, $\dim C_G(g) \leqs 28$ and we have reached a contradiction.

Now consider the case $M^0 = C_4$, with $p \ne 2$. Note that $M = C_G(\tau)$ is connected, where $\tau$ is a graph automorphism of $G$. First assume $g$ is unipotent. By inspecting \cite{Law09}, we calculate that $\a(g) = 2/3$ if $g=u_{\a}$, otherwise $\b(g) \geqs 18$ and thus $\a(g) \leqs 4/7$. Now suppose $g$ is semisimple. If $\dim g^G \geqs 48$ then the trivial bound $\dim(g^G \cap M) \leqs 32$ yields 
$\a(g)<2/3$, so it remains to consider the following elements:
\[
\begin{array}{cccc} \hline
C_G(g) & D_5T_1 & A_5A_1 & A_5T_1 \\
\dim g^G & 32 & 40 & 42 \\ \hline
\end{array}
\] 
If $C_G(g) = A_5A_1$ then $g$ is an involution and thus $\dim (g^G \cap M) \leqs 20$. This gives $\dim X(g) \leqs 22$. Next assume $C_G(g) = A_5T_1$. As explained in the proof of \cite[Lemma 6.2]{LLS}, we have $\dim C_M(g) \geqs |\Sigma^{+}(A_5)| = 15$, where $\Sigma^{+}(A_5)$ is the set of positive roots in a root system of type $A_5$. Once again, this implies that $\dim (g^G \cap M) \leqs 20$ (note that $\dim(g^G \cap M)$ is even). Finally, suppose $C_G(g) = D_5T_1$. Here $\dim C_M(g) \geqs |\Sigma^{+}(D_5)| = 20$ and thus $\dim (g^G \cap M) \leqs 16$. This yields 
$\dim X(g) \leqs 26$ and $\a(g) \leqs 13/21$.

Next assume $M^0=F_4$, so $M$ is connected and $\dim X = 26$. As in the previous case, we have $M = C_{G}(\tau)$ for a graph automorphism $\tau$. Suppose $g$ is unipotent. By \cite{Law09}, if $g=u_{\a}$ 
then $\dim g^G = 22$ and $\dim (g^G \cap M) = 16$, so $\dim X(g)=20$ and $\a(g) = 10/13$. In each of the remaining cases, $\b(g) \geqs 10$ and thus $\a(g) \leqs 8/13$. Now assume $g$ is semisimple. If $\dim g^G \geqs 58$ then the trivial bound $\dim(g^G \cap M) \leqs 48$ implies that $\dim X(g) \leqs 16$ and the result follows. The possibilities with $\dim g^G < 58$ are as follows:
\[
\begin{array}{ccccccccc} \hline
C_G(g) & D_5T_1 & A_5A_1 & A_5T_1 & D_4T_2 & A_4A_1T_1 & A_4T_2 & A_2^3 & A_3A_1^2T_1 \\ 
\dim g^G & 32 & 40 & 42 & 48 & 50 & 52 & 54 & 56 \\ \hline
\end{array}
\]
Suppose $C_G(g) = A_3A_1^2T_1$. As noted in the proof of \cite[Lemma 6.2]{LLS}, we have $\dim C_M(g) \geqs |\Sigma^{+}(A_3A_1^2)| = 8$ so $\dim(g^G \cap M) \leqs 44$ and thus $\a(g) < 2/3$. The case $C_G(g) = A_4T_2$ is entirely similar. If $C_G(g) = A_2^3$ then $g$ has order $3$ and thus $\dim(g^G \cap M) \leqs 36$ (see \cite[Table 4.7.1]{GLS}, for example). This implies that $\a(g)<2/3$. In each of the five remaining cases, $C_M(g)^0$ is determined in the proof of \cite[Lemma 6.2]{LLS} and the required bound quickly follows.

Finally, let us assume $M^0 = A_2G_2$, so $M$ is connected and $\dim X = 56$. First assume $g$ is unipotent. By inspecting \cite{Law09}, we deduce that $\dim(g^G \cap M) = 6$ if $g=u_{\a}$, which gives $\a(g) = 5/7$. In all other cases, $\b(g) \geqs 28$ and thus $\a(g) \leqs 1/2$. Now assume $g$ is semisimple of order $r$ and note that $\dim (g^G \cap M) \leqs 18$. If $\dim g^G \geqs 40$ then $\dim X(g) \leqs 34$, so we may assume $\dim g^G < 40$, which means that $C_G(g) = D_5T_1$. We claim that $\dim (g^G \cap M) \leqs 12$, which gives $\a(g) \leqs 9/14$. This is clear if $g$ is an involution, so let us assume $r$ is odd. We have
\[
V{\downarrow}M = {\rm Lie}(A_2G_2) \oplus (V_8 \otimes V_7),
\]
where $V_8$ is the Lie algebra of $A_2$ and $V_7 = W_{G_2}(\l_1)$ (see \cite[Chapter 12]{Thomas}). Let $W$ be the summand $V_8 \otimes V_7$. Seeking a contradiction, suppose there exists $y = y_1y_2 \in g^G \cap M$ with $\dim y^M > 12$. Then $y_1$ and $y_2$ are nontrivial and the $1$-eigenspace of $y$ on ${\rm Lie}(A_2G_2)$ is at most $8$-dimensional (since $\dim C_M(y) \leqs 8$). Define $\nu(y,W)$ and $\nu(y_2,V_7)$ as in \eqref{e:nudef}. Since $\nu(y_2,V_7) \geqs 4$, it follows that 
$\nu(y,W) \geqs 8.4=32$ and thus the $1$-eigenspace of $y$ on $W$ is at most $24$-dimensional. But this implies that $\dim C_G(y) \leqs 32$ and we have reached a contradiction. 
\end{proof}

\begin{lem}\label{l:nmrf4}
The conclusion to Proposition \ref{t:nmr} holds if $G = F_4$.
\end{lem}

\begin{proof}
The cases we need to consider are as follows (in each case, $M$ is connected):
\[
\begin{array}{cccc} \hline
M & A_1 \, (p \geqs 13) & G_2 \, (p=7) & A_1G_2 \, (p \ne 2) \\
\dim X & 49 & 38 & 35 \\ \hline
\end{array}
\]
Let $g \in M$ be an element in $\mathcal{P}$. The result for unipotent elements quickly follows from the fusion computations in \cite{Law09}. More precisely, we get $\a(g) < 2/3$ unless $M = A_1G_2$ and $g=u_{\a}$, in which case $\dim g^G = 16$ and $\dim (g^G \cap M) = 6$, so $\a(g) = 5/7$. For the remainder, we may assume $g$ is semisimple. If $\dim g^G \geqs 28$ then the trivial bound $\dim(g^G \cap M) \leqs \dim M - {\rm rank}\,M$ is good enough, so let us assume $\dim g^G< 28$. This means that $p \ne 2$ and $g$ is an involution with $C_G(g) = B_4$.

Suppose $M = A_1$ (with $p \geqs 13$). We claim that $M$ does not contain any $B_4$-involutions. To see this, let $V$ be the Lie algebra of $G$ and observe that 
\[
V{\downarrow}M = {\rm Lie}(A_1) \oplus W_{A_1}(22\l_1) \oplus W_{A_1}(14\l_1) \oplus W_{A_1}(10\l_1)
\]
(see \cite[Chapter 12]{Thomas}). Here $M$ is the adjoint group and we can use this decomposition to compute the eigenvalues on $V$ of an involution $g=[-i,i] \in M$ (note that $M$ contains a unique class of involutions). One can check that $g$ acts as  $[-I_{28},I_{24}]$ on $V$, so $\dim C_G(g) = 24$ and thus $C_G(g) = A_1C_3$. This justifies the claim. Similarly, there are no $B_4$-type involutions when $M=G_2$ (with $p=7$). 

Finally, let us assume $M = A_1G_2$. Here $V{\downarrow}M = {\rm Lie}(A_1G_2) \oplus (V_5 \otimes V_7)$, where $V_5 = W_{A_1}(4\l_1)$ and $V_7 = W_{G_2}(\l_1)$ (see \cite[Chapter 12]{Thomas}). Let $y=y_1y_2 \in M$ be an involution.
Using the above decomposition, we calculate that $y$ is a $B_4$-involution if and only if $y_2=1$. Therefore $\dim(g^G \cap M) = 2$ and $\dim X(g) = 21$, which gives $\a(g) = 3/5$.
\end{proof}

\begin{lem}\label{l:nmrg2}
The conclusion to Proposition \ref{t:nmr} holds if $G = G_2$.
\end{lem}

\begin{proof}
Here $M = A_1$ and $p \geqs 7$, so $\dim X = 11$. Let $g \in M$ be an element in 
$\mathcal{P}$. If $g$ is unipotent, then \cite{Law09} implies that $g$ is regular, so $\dim g^G = 12$ and thus $\dim X(g) = 1$. On the other hand, if $g$ is semisimple then $\dim g^G \geqs 6$ and $\dim (g^G \cap M) = 2$, so $\dim X(g) \leqs 7$ and the result follows.
\end{proof}

\vs

This completes the proof of Proposition \ref{t:nmr}. By combining this with Proposition \ref{t:mr} and the results in Section \ref{s:parab}, we conclude that the proof of Theorem \ref{t:mainex} is complete. In particular, we have now established Theorem \ref{t:main4}.  

\vs

As noted in Section \ref{s:intro}, Theorem \ref{t:main5} follows immediately, as does Theorem \ref{t:main5s}. The next result shows that the bounds presented in Theorem \ref{t:main5} are best possible.

\begin{thm}\label{t:best}   
Let $G$ be a simple exceptional algebraic group over an algebraically closed field and set $t=3$ if $G =G_2$ and $t=4$ in all other cases. If $x_1, \ldots, x_t$ are long root elements in $G$, then $\la x_1, \ldots, x_t \ra$ is not dense in $G$.
\end{thm}

\begin{proof}  
Let $g \in G$ be a long root element. In every case, we produce a finite dimensional $G$-module $V$ such that $V^G=0$ (see \eqref{e:VG}) and $\dim C_V(g) > \frac{3}{4}\dim V$ (or $\dim C_V(g) > \frac{2}{3}\dim V$ for $G = G_2$). It follows that any $t$ conjugates of $g$ have a common nontrivial fixed space on $V$ and so they do not topologically generate $G$. 

For $G=G_2$, we take $V$ to be the $7$-dimensional Weyl module $W_G(\l_1)$, so $V$ is irreducible if $p \ne 2$ and it is indecomposable with $V^G = 0$ when $p=2$. By inspecting \cite[Table 1]{Lawunip}, we see that $\dim C_V(g) = 5$ and the result follows. Similarly, if $G=F_4$ then we set $V = V_G(\l_4)$, so $\dim V = 26 - \delta_{3,p}$ and \cite[Table 3]{Lawunip} yields $\dim C_V(g) \geqs 20-\delta_{3,p} > \frac{3}{4}\dim V$. Finally, for $E_6$, $E_7$ and $E_8$, we take $V$ to be the smallest irreducible restricted module (of dimensions $27$, $56$ and $248$, respectively) and once again, by inspecting \cite{Lawunip}, we deduce that $\dim C_V(g) > \frac{3}{4}\dim V$.  
\end{proof}  

\begin{rem}\label{r:finitebest} 
Let $G(q)$ be a finite quasisimple exceptional group of Lie type over $\mathbb{F}_q$ and let $V$ be the $G$-module defined in the proof of Theorem \ref{t:best}. Since $G(q)$ acts irreducibly on $V$ (or indecomposably in the case of $G_2(q)$ in characteristic $2$), the above proof shows that if $G(q)$ contains a long root element $g \in G$, then $G(q)$ is not generated by any $t$ conjugates of $g$.
Of course, in almost all cases $G(q)$ does indeed contain long root elements of $G$ (this fails for the Suzuki and Ree groups). 
\end{rem} 

\subsection{Generic stabilizers}\label{ss:generic}

Finally, we prove Theorem \ref{t:generic}. Let $G$ be a simple exceptional algebraic group over an algebraically closed field
$k$ of characteristic $p \geqs 0$.  Let $V$ be a faithful rational $kG$-module
and assume that $V^G=0$.  Recall that the generic stabilizer of $G$ on $V$ is trivial if there is a nonempty open subset $V_0$ of $V$ such that the stabilizer $G_v$ is trivial for all $v \in V_0$. Also recall that $d(G) = 3(\dim G - {\rm rank}\,G)$ is as follows:
\[
\begin{array}{cccccc} \hline
G & E_8 & E_7 & E_6 & F_4 & G_2  \\ 
d(G) &  720  & 378 &  216  &  144  & 36 \\ \hline
\end{array}
\]
For $g \in G$, we will write $V(g) = \{v \in V \,:\, v^g = v\}$ for the fixed space of $g$ on $V$.

As in \cite{GG, GL}, if the inequality
\begin{equation}\label{e:bdd}
\dim V(g) + \dim g^G < \dim V
\end{equation}
holds for all $g \in G$ of prime order (and all nontrivial unipotent elements if $p=0$), then the generic stabilizer of $G$ on $V$ is trivial.

First assume that $g$ is not one of the exceptions listed in parts (ii) and (iii) of Theorem \ref{t:main5s}.  Then $G$ is (topologically) generated by three conjugates of $g$, whence $\dim V(g) \leqs \frac{2}{3}\dim V$. Moreover, since $\dim g^G \leqs \dim G - {\rm rank}\, G$, we deduce that the inequality in \eqref{e:bdd} holds whenever $\dim V > 3(\dim G - {\rm rank} \, G) = d(G)$.
Similarly, in the exceptional cases we see that \eqref{e:bdd} holds as long as $\dim V >  5 \dim g^G$, and this bound is satisfied since $\dim V > d(G)$. 

\vs

This completes the proof of Theorem \ref{t:generic}.

\section{Random generation of finite exceptional groups}\label{s:random}

In this final section we prove Theorems \ref{t:main6} and \ref{t:main7}. We begin by considering Theorem \ref{t:main6}; the two parts in the statement will be handled separately in Sections \ref{ss:thm6_1} and \ref{ss:thm6_2}, respectively.

\subsection{Proof of Theorem \ref{t:main6}(i)}\label{ss:thm6_1}

Let $G$ be a simply connected simple algebraic group over the algebraic closure of a finite field of characteristic $p$. Let us assume $G$ is one of the following 
\begin{equation}\label{e:gros}
E_8, \, E_7, \, E_6, \, F_4, \, G_2, \, D_4, \, B_2 \, (p=2)
\end{equation}
and fix a Steinberg endomorphism $\s$ of $G$ such that $G_{\s} = G(q)$ is a finite quasisimple exceptional group of Lie type over $\mathbb{F}_q$, where $q=p^f$ for some $f \geqs 1$. 

Let $r$ be a prime and recall that
\begin{align*}
\mathcal{C}(G,r,q) & = \max\{\dim g^G \,:\, \mbox{$g \in G(q)$ has order $r$ modulo $Z(G)$}\} \\
\gamma(G,r) & = \left\{\begin{array}{ll} 
\dim G_{[r]} & \mbox{if $r=p$ or $r \in\{2,3\}$} \\
\ell(G) & \mbox{otherwise} 
\end{array}\right.
\end{align*}
with $\ell(G)$ defined as follows:
\[
\begin{array}{cccccccc} \hline
G & E_8 & E_7 & E_6 & F_4 & G_2 & D_4 & B_2 \\
\ell(G) & 200 & 100 & 58 & 40 & 10 & 24-2\delta_{5,r} & 8 \\ \hline
\end{array}
\]

Here $G_{[r]}$ is the subvariety of elements $g \in G$ with $g^r \in Z(G)$. The dimension of $G_{[r]}$ is computed in \cite{Law05} and we record the values for $r<h$ in Table \ref{tab:gr}, where $h$ denotes the Coxeter number of $G$ (recall that if $r \geqs h$, then $\dim G_{[r]} = \dim G - {\rm rank}\,G$). 

\begin{rem}
Clearly, if $r$ does not divide $|Z(G)|$ (in particular, if $r=p$ or $r \geqs 5$), then 
\[
\mathcal{C}(G,r,q) = \max\{\dim g^G \,:\, \mbox{$g \in G(q)$ has order $r$}\}.
\]
In fact, the same conclusion holds in all cases unless  $G=E_7$, $p \ne 2$ and $r=2$. In this special case, the adjoint group 
has three classes of involutions but two of
the classes contain involutions that only lift to elements $g \in G$ of order $4$ with $C_G(g) = A_7$ or $E_6T_1$, whereas every non-central involution in $G(q)$ has centralizer in $G$ of type $A_1D_6$. 
\end{rem}

In Table \ref{tab:class} we record the conjugacy classes of elements $g \in G$ of order $r \in \{2,3\}$ (modulo $Z(G)$) with $\dim g^G = \dim G_{[r]}$. Let us comment on the notation in this table. For $r \ne p$ we give the structure of $C_G(g)$ and for $r=p$ and $G$ exceptional we use the standard labelling of unipotent classes from \cite{LS_book}. Finally, for unipotent elements when $G$ is classical we use the notation from \cite{AS} if $p=2$  and we give the Jordan form of $g$ on the natural module when $p=3$.

\renewcommand{\arraystretch}{1.1}
\begin{table}
\begin{center}
\[
\begin{array}{l|cccccccccc}
 &  2 & 3 & 5 & 7 & 11 & 13 & 17 & 19 & 23 & 29  \\ \hline
E_8 & 128 & 168 & 200 & 212 & 224 & 228 & 232 & 234 & 236 & 238 \\
E_7 & 70  & 90 & 106 & 114 & 120 & 122 & 124 &  & &   \\
E_6 & 40 & 54 & 62 & 66 & 70 & & & & &  \\
F_4 & 28 & 36 & 40 & 44 & 46 & & & & & \\
G_2 & 8 & 10 & 10 & & & &  & & & \\
D_4 & 16 & 18 & 22 & & & & & & & \\
B_2 & 6 & 6 & & & & & & & &  
\end{array}
\]
\caption{The dimension of $G_{[r]}$ for $r<h$}
\label{tab:gr}
\end{center}
\end{table}
\renewcommand{\arraystretch}{1}

\renewcommand{\arraystretch}{1.1}
\begin{table}
\begin{center}
\[
\begin{array}{ccccccccc} \hline
r & p & E_8 & E_7 & E_6 & F_4 & G_2 & D_4 & B_2 \\ \hline
2 & \ne 2 & D_8 & A_7 &  A_1A_5 & A_1C_3 & A_1^2 & A_1^4 & - \\
2 & 2 & A_1^4 & A_1^4 & A_1^3 & A_1\tilde{A}_1 & \tilde{A}_1 & c_4 & c_2 \\
3 & \ne 3 & A_8 & A_2A_5 & A_2^3 & A_2^2 & A_1T_1 & A_1^3T_1 \mbox{ or } A_2T_2 & A_1T_1 \\
3 & 3 & A_2^2A_1^2 & A_2^2A_1 & A_2^2A_1 & \tilde{A}_2A_1 & G_2(a_1) & [J_3^2,J_1^2] & - \\ \hline
\end{array}
\]
\caption{The classes with $\dim g^G = \dim G_{[r]}$, $r = 2,3$}
\label{tab:class}
\end{center}
\end{table}
\renewcommand{\arraystretch}{1}

\begin{prop}\label{p:cgr1}
If $r=p$, then $\mathcal{C}(G,r,q) = \dim G_{[r]}$.
\end{prop}

\begin{proof}
To begin with, let us assume $G(q)$ is one of the following twisted groups:
\[
{}^2B_2(q),\,  {}^2G_2(q), \, {}^2F_4(q), \, {}^3D_4(q).
\]
Suppose $G(q) = {}^2G_2(q)$, so $r=p=3$. From \cite[Table 22.2.7]{LS_book} we see that the largest class of elements of order $3$ in $G(q)$ is contained in the $G$-classes labelled $G_2(a_1)$, whence $\mathcal{C}(G,3,q) = 10 =\dim G_{[3]}$. Similarly, $G(q) = {}^2F_4(q)$ has two classes of involutions; the largest one is in the $G$-class $A_1\tilde{A}_1$ (see \cite[Table 22.2.5]{LS_book}) and thus $\mathcal{C}(G,2,q)=28$. Next assume $G(q) = {}^2B_2(q)$, so $G$ is of type $B_2$ and $r=p=2$. The largest class of involutions in $G$ comprises the elements of type $c_2$ (in terms of the notation of \cite{AS}), so $\dim G_{[2]} = 6$. Moreover, this class is $\s$-stable (note that $\s$ interchanges the other two classes of involutions in $G$), whence $\mathcal{C}(G,2,q) = 6$. Finally, suppose $G(q) = {}^3D_4(q)$. If $p=2$ then the largest class of involutions in $G=D_4$ has dimension $16$; in the notation of \cite{AS}, these are the elements of type $c_4$ and this class is $\s$-stable (see \cite[Proposition 3.55]{Bur2}, for example), so $\mathcal{C}(G,2,q)=16$. Similarly, if $p=3$ or $5$ then the elements of order $p$ in the largest class in $G$ have Jordan form $[J_3^2,J_1^2]$ and $[J_5,J_3]$, respectively, on the natural module. Both of these classes are $\s$-stable, so $\mathcal{C}(G,3,q) = 18$ and $\mathcal{C}(G,5,q) = 22$. Finally, if $p \geqs 7$ then $G(q)$ contains regular unipotent elements of order $p$ and we deduce that $\mathcal{C}(G,p,q) = 24$.

In each of the remaining cases, we observe that every unipotent class in $G$ is $\s$-stable and therefore has representatives in $G(q)$ (see \cite[Section 20.5]{LS_book}). The result follows.
\end{proof}

In the next two propositions, we assume $r$ is a prime divisor of $|G(q)|$. In particular, $r \ne 3$ if $G = B_2$.

\begin{prop}\label{p:cgr2}
If $r \ne p$ and $r \in \{2,3\}$ then $\mathcal{C}(G,r,q) = \dim G_{[r]}$.
\end{prop}

\begin{proof}
First assume $r=2$, so $q$ is odd. If $G$ is of type $D_4$ then $\dim G_{[2]}=16$ and the largest class of involutions consists of elements of the form $[-I_4, I_4]$ (with centralizer of type $A_1^4$). Moreover, this class is $\s$-stable (see \cite[Proposition 3.55]{Bur2}) and the desired result follows. In the remaining cases, by inspecting \cite[Tables 4.3.1 and 4.5.1]{GLS}, we deduce that every conjugacy class of involutions in $G$ is defined over $\mathbb{F}_q$. In addition, if $G = E_7$ and $p \ne 2$ then there are elements $g \in G(q)$ of order $4$ (and order $2$ modulo $Z(G)$) with $C_G(g) = A_7$.

Now assume $r=3$. The elements of order $3$ in the largest class in $D_4$ are of type $[I_4, \omega I_2, \omega^2I_2]$ or $[I_2, \omega I_3, \omega^2I_3]$, where $\omega \in k$ is a primitive cube root of unity, and we observe that both classes are $\s$-stable. This gives the result for $G(q) = {}^3D_4(q)$ and we note that ${}^2B_2(q)$ does not contain elements of order $3$. We now complete the proof by inspecting \cite[Tables 4.7.1 and 4.7.3A]{GLS}.
\end{proof}

\begin{prop}\label{p:cgr3}
If $r \ne p$ and $r \geqs 5$ then $\mathcal{C}(G,r,q) \geqs \gamma(G,r)$.
\end{prop}

\begin{proof}
Let $g \in G(q)$ be an element of order $r$, where $r \geqs 5$ is a prime and $r \ne p$. Note that $C_G(g) = HT$ is a connected reductive group, where $H$ is a semisimple subsystem subgroup of $G$ and $T = Z(C_G(g))^0$ is a central torus. Set $d = \dim T$ and let $e$ be the order of $q$ modulo $r$ (so $e$ is the smallest positive integer such that $r$ divides $q^e-1$).

If $G = B_2$ then it is easy to see that $g$ is regular as an element of $G$ (see \cite[Proposition 3.52]{Bur2}, for example), so $\dim g^G = 8 = \dim G_{[r]}$. Similarly, if $G = G_2$ then $C_G(g) = A_1T_1$ or $T_2$, and thus $\mathcal{C}(G,r,q) \geqs 10$. 

Next assume $G = F_4$. If $G(q) = {}^2F_4(q)$, then \cite[Table IV]{Shinoda1} indicates that $C_G(g) = B_2T_2$, $\tilde{A}_{1}A_1T_2$ or $T_4$ and we deduce that $\dim g^G \geqs 40$. 
Now assume $G(q) = F_4(q)$. If $e \in \{1,2\}$ then by inspecting \cite{Lub} we deduce that there exists an element $g \in G(q)$ of order $r$ such that $\tilde{A}_2A_1T_1 \leqs C_G(g)$. Since $C_G(g) = HT$ as described above, and since $\tilde{A}_2A_1T_1$ is not contained in $B_3T_1$ or $C_3T_1$, it follows that $C_G(g) = \tilde{A}_2A_1T_1$ and thus $\mathcal{C}(G,r,q) \geqs 40$. Now assume $e \geqs 3$, so $d \geqs 2$. If $d \geqs 3$ then $\dim C_G(g)$ is at most $\dim A_1T_3 = 6$ and thus $\dim g^G \geqs 46$, so we may assume $d=2$ and $e \in \{3,4,6\}$. By inspecting \cite{Lub}, we see that $C_G(g) = B_2T_2$ if $e=4$, otherwise $C_G(g) = A_2T_2$ (or $\tilde{A}_2T_2$). The result follows.

Now let us consider the case $G=E_6$. First assume $e \in \{1,2\}$. By \cite{Lub}, there exists $g \in G(q)$ of order $r$ with $A_2^2A_1T_1 \leqs C_G(g)$ and we claim that equality holds. Suppose otherwise. Then $C_G(g)$ must be one of $A_5T_1$, $D_5T_1$ or $A_4A_1T_1$. But $A_2^2A_1$ is not contained in $A_5$, $D_5$ or $A_4A_1$, so all three possibilities can be ruled out. This justifies the claim and we deduce that $\mathcal{C}(G,r,q) \geqs 58$. Now assume $e \geqs 3$ and note that $d \geqs 2$. If $d \geqs 3$ then $\dim C_G(g) \leqs \dim A_3T_3$ and the result follows, so let us assume $d=2$. By \cite{Lub}, we deduce that $e \in \{3,6\}$ (if $e=4$ then $d \geqs 3$) and we can choose $g \in G(q)$ of order $r$ with $A_2^2T_2 \leqs C_G(g)$. Since $A_2^2T_2$ is not contained in $A_1^4T_2$ or $D_4T_2$, we conclude that $C_G(g) = A_2^2T_2$ and $\dim g^G = 60$.

Next suppose $G=E_7$. If $e \in \{1,2\}$ then \cite{Lub} shows that we can choose $g \in G(q)$ of order $r$ such that $A_3A_2A_1T_1 \leqs C_G(g)$. Therefore, $d=1$ and either $C_G(g) = A_3A_2A_1T_1$, or $C_G(g)$ is one of $A_6T_1$, $D_6T_1$, $E_6T_1$, $D_5A_1T_1$ or 
$A_4A_2T_1$. But $A_3A_2A_1$ is not contained in the semisimple part of any of these groups, whence $C_G(g) = A_3A_2A_1T_1$ is the only option and thus $\dim g^G = 106$. Now assume $e \geqs 3$. If $d \geqs 3$ then $\dim C_G(g) \leqs \dim D_4T_3$ and thus $\dim g^G \geqs 102$. Finally, suppose $d=2$. If $e=4$ then we can choose $g \in G(q)$ of order $r$ such that $D_4A_1T_2 \leqs C_G(g)$ and by considering the possibilities for $C_G(g)$ with $d=2$ it is easy to check that $C_G(g) = D_4A_1T_2$ is the only option, so $\dim g^G = 100$. Finally, if $e \in \{3,6\}$ then we choose $g \in G(q)$ with $A_2A_1^3T_2 \leqs C_G(g)$ and in the usual manner, using \cite{Lub}, we deduce that $C_G(g) = A_2A_1^3T_2$ and $\dim g^G = 114$.

Now assume $G = E_8$. If $r=5$ then we can choose $g \in G(q)$ of order $r$ with $C_G(g) = A_4^2$, which gives $\dim g^G = 200$. For the remainder, let us assume $r \geqs 7$. If $e \in \{1,2\}$ then $q \geqs 8$ (since $r \geqs 7$) and by considering \cite{Lub} we see that there exists $g \in G(q)$ of order $r$ with $J \leqs C_G(g)$, where 
\[
J= \left\{\begin{array}{ll}
D_4A_3T_1 & \mbox{if $q$ is odd} \\
A_5A_2T_1 & \mbox{if $q$ is even.}
\end{array}\right.
\]
By inspecting the possibilities for $C_G(g)$ with $d=1$ we deduce that $C_G(g) = J$ and thus $\dim g^G = 204$. Now assume $e \geqs 3$. If $d \geqs 3$ then $\dim C_G(g) \leqs \dim D_5T_3$, which gives $\dim g^G \geqs 200$ as required. Finally, suppose $d=2$. If $e=4$ then $q \geqs 4$ (since $r \geqs 7$) and by inspecting \cite{Lub} we see that there exists $g \in G(q)$ of order $r$ with $A_2^2A_1^2T_2 \leqs C_G(g)$. In the usual way, we conclude that $C_G(g) = A_2^2A_1^2T_2$, which gives 
$\dim g^G = 224$. Similarly, if $e \in \{3,6\}$ then we can choose $g \in G(q)$ with $C_G(g) = D_4A_2T_2$ and the result follows.

To complete the proof of the proposition, we may assume $G(q) = {}^3D_4(q)$. Here $\s = \varphi\tau$, where $\varphi$ is a standard Frobenius morphism (corresponding to the field automorphism $\l \mapsto \l^q$) and $\tau$ is a triality graph automorphism of $G$. Let $\omega \in k$ be a primitive $r$-th root of unity. For $r=5$ one checks that the $G$-class represented by $[I_2, \omega I_2, \omega^4I_2, \omega^2, \omega^3]$ is $\s$-stable, so there is a representative in $G(q)$ and thus $\mathcal{C}(G,5,q) = 22 = \dim G_{[5]}$. Now assume $r \geqs 7$. If $e \in \{1,2\}$ then the regular class represented by the element $[I_2, \omega, \omega^{-1}, \omega^2, \omega^{-2}, \omega^3, \omega^{-3}]$ is $\s$-stable and so we have $\mathcal{C}(G,r,q) = 24 = \dim G_{[r]}$. Similarly, if $e \in \{3,6\}$ then $r$ divides $q^2+\a q+1$ (where $\a=1$ if $e=3$, otherwise $\a=-1$) and one checks that the $G$-class represented by $[I_2, \omega, \omega^{\a q}, \omega^{q^2}, \omega^{-1}, \omega^{-\a q}, \omega^{-q^2}]$ is $\s$-stable. Finally, suppose $e=12$, so $r$ divides $q^4-q^2+1$. Here it is convenient to view ${}^3D_4(q)$ as the centralizer of a triality graph-field automorphism $\psi$ of $\O_{8}^{+}(q^3)$. Now the order of $q^3$ modulo $r$ is $4$ and it is straightforward to check that every $\psi$-stable conjugacy class of elements of order $r$ in $\O_{8}^{+}(q^3)$ is regular.
\end{proof}

This completes the proof of Theorem \ref{t:main6}(i).

\vs

\subsection{Proof of Theorem \ref{t:main6}(ii)}\label{ss:thm6_2}

Now let us turn to part (ii) of Theorem \ref{t:main6}. As before, $G$ is a simply connected simple algebraic group as in \eqref{e:gros} and $\s$ is a Steinberg endomorphism such that $G_{\s} = G(q)$ is a finite quasisimple exceptional group of Lie type over $\mathbb{F}_q$, where $q=p^f$. In addition, set $f(r) = \frac{2+\delta_{2,r}}{5}$.

\begin{prop}\label{p:par1}
Let $r$ be a prime divisor of $|G(q)|$ and let $g \in G(q)$ be an element of order $r$ modulo $Z(G)$ with $\dim g^G \geqs \gamma(G,r)$. Then $\a(G,M,g) <f(r)$ for every maximal parabolic subgroup $M$ of $G$.
\end{prop}

\begin{proof}
First assume $r \ne p$ and set $D=C_G(g)$. If $G$ is an exceptional algebraic group, then we can use the upper bound on $\dim X(g)$ given by \cite[Theorem 2(I)(b)]{LLS}. For example, suppose $G=E_8$, $M=P_1$ and $r \geqs 5$. Now $\dim g^G \geqs 200$, so $D$ does not have a simple factor of type $D_8$ or $E_7$. Therefore, by inspecting \cite[Table 7.3]{LLS}, we deduce that $\dim X - \dim X(g) \geqs 48$, which gives $\dim X(g) \leqs 30$ and $\a(G,M,g) \leqs \frac{5}{13}$. 

The reader can check that the bound supplied by \cite[Theorem 2(I)(b)]{LLS} is sufficient unless we are in one of the following cases:
\begin{itemize}\addtolength{\itemsep}{0.2\baselineskip}
\item[{\rm (a)}] $G=G_2$, $M=P_1$ or $P_2$, $r \geqs 3$;
\item[{\rm (b)}] $G=E_6$, $M=P_2$, $r \geqs 3$.
\end{itemize}

Write $M=QL$, where $Q=R_u(M)$ is the unipotent radical and $L$ is a Levi factor. Without loss of generality, we may assume that $g \in M$ and $\dim(g^G \cap M) = \dim g^M$. By \cite[Lemma 3.1]{LLS}, $D \cap M$ is a parabolic subgroup of $D$ with $R_u(D \cap M) \leqs Q$ and 
\begin{equation}\label{e:ss}
\dim X(g) = \dim R_u(D \cap M).
\end{equation}
In case (a), we have $\dim X = 5$ and $D = A_1T_1$ or $T_2$. Therefore, \eqref{e:ss} implies that $\dim X(g) \leqs 1$ and the result follows.

Now let us turn to case (b). Here $\dim X = 21$, $Q = U_{21}$ (here $U_m$ denotes a unipotent group of dimension $m$) and $L=A_5T_1$, so it suffices to show that $\dim R_u(D \cap M) \leqs 8$. If $r \geqs 5$ then $\dim g^G \geqs 58$ and so $D$ has at most $7$ positive roots, which gives $\dim R_u(D \cap M) \leqs 7$. Finally, suppose $r=3$. Here $D = A_2^3$ has $9$ positive roots and it is easy to see that at least one of the corresponding root subgroups is not contained in $Q$. Indeed, if we fix a set of simple roots $\{\a_1, \ldots, \a_6\}$ for $G$, then $Q = \la U_{\a} \, : \, \a \in S \ra$, where $S$ is the set of all $21$ positive roots of the form $\sum_{i}c_i\a_i$ with $c_2 \ne 0$. But $D$ has at least one positive root whose $\a_2$ coefficient is zero (since $g$ is not regular in the Levi factor
$L=A_5T_1$, it centralizes both positive and negative root subgroups in $A_5<L$), whence $\dim R_u(D \cap M) \leqs 8$.

To complete the analysis of semisimple elements, we may assume $G=B_2$ or $D_4$. If $G = B_2$ then $r \geqs 5$ and $g$ is regular, so \eqref{e:ss} implies that $\dim X(g)=0$. 

Suppose $G=D_4$. Here we may assume that $M=P_1$ or $P_2$, where $P_1 = U_6A_3T_1$ and $P_2 = U_9A_1^3T_1$. Note that $\dim X = 6$ if $M=P_1$ and $\dim X = 9$ if $M=P_2$. If $r \geqs 5$ then $C_G(g) = A_1T_3$ or $T_4$ and thus \eqref{e:ss} gives $\dim X(g) \leqs 1$. Now assume $r=3$, so $\dim g^G = 18$ and there are two possibilities for $D$, namely $D = A_1^3T_1$ and $A_2T_2$. In both cases, $D$ has $3$ positive roots, so $\dim R_u(D \cap M) \leqs 3$ and \eqref{e:ss} yields $\dim X(g) \leqs 3$. This gives the desired result for $M=P_2$, but further work is needed when $M=P_1$.

Suppose $M=P_1$ and note that we may identify $X = G/M$ with the set of totally singular $1$-spaces in the natural module $V$ for $G$. Let 
\begin{equation}\label{e:basis}
\mathcal{B} = \{e_1, \ldots, e_4, f_1, \ldots, f_4\}
\end{equation}
be a standard orthogonal basis for $V$ (with respect to the quadratic form preserved by $G$) and let $\omega \in k$ be a primitive cube root of unity. If $D = A_2T_2$ then we may assume $g = [I_2, \omega I_3, \omega^2I_3]$ has eigenspaces $\la e_1,f_1\ra$, $\la e_2,e_3,e_4\ra$ and $\la f_2,f_3,f_4\ra$. Now $U \in X$ is fixed by $g$ if and only if $U$ is contained in one of the eigenspaces of $g$ and it follows that $\dim X(g)=2$ since the Grassmannian ${\rm Gr}(1,k^3)$ is $2$-dimensional. Similarly, if $D=A_1^3T_1$ then $g = [I_4, \omega I_2,\omega^2I_2]$ has eigenspaces $\la e_1,f_1,e_2,f_2\ra$, $\la e_3,e_4\ra$ and $\la f_3,f_4\ra$. Once again we deduce that $\dim X(g) = 2$ since the variety of totally singular $1$-spaces contained in the nondegenerate $1$-eigenspace $\la e_1,f_1,e_2,f_2\ra$ is $2$-dimensional.

Finally, let us assume $G=D_4$, $r=2$ and $p \ne 2$. Here $D=A_1^4$ and thus 
$\dim X(g)  = \dim R_u(D \cap M) \leqs 4$ by \eqref{e:ss}. As before, we need a stronger upper bound when $M=P_1$. Here we may assume that $g = [-I_4,I_4]$ has eigenspaces $\la e_1,f_1,e_2,f_2\ra$ and $\la e_3,f_3,e_4,f_4\ra$, and by arguing as in the previous case we deduce that $\dim X(g)=2$.

To complete the proof of the proposition, we may assume that $r=p$. Set $D=C_G(g)$ and let $\mathcal{B}_g$ be the variety of Borel subgroups of $G$ containing $g$. By
combining \cite[Proposition 1.9]{LLS} with \cite[Lemma 2.2]{LLS} we get
\begin{equation}\label{e:unip}
\dim X(g) \leqs \dim \mathcal{B}_g = \frac{1}{2}(\dim D - {\rm rank}\, G).
\end{equation}
In particular, notice that $\dim X(g)=0$ if $g$ is regular.

First assume $G$ is exceptional. If $r \geqs 5$, then it is easy to check that the upper bound in \eqref{e:unip} is sufficient. For example, if $G=E_7$ then $\dim g^G \geqs 106$, so $\dim D \leqs 27$ and thus \eqref{e:unip} gives $\dim X(g) \leqs 10$. This is sufficient since $\dim X \geqs 27$ (see Table \ref{t:parab}). Now assume $r \in \{2,3\}$. If $G=G_2$ then the same approach is effective, but for the other exceptional groups there are cases where the bound in \eqref{e:unip} is insufficient. Specifically, we need to establish a better upper bound in the following cases:
\begin{align*}
E_8{:} & \;\; P_1, P_2 \, (r=2), P_7, P_8 \\
E_7{:} & \;\; P_1, P_2, P_6, P_7 \\
E_6{:} & \;\; P_1, P_2, P_3\, (r=2), P_5 \, (r=2), P_6 \\
F_4{:} & \;\; P_1, P_4
\end{align*}

Let us assume $G(q)$ is untwisted and consider the action of $G(q)$ on the set of cosets of $M(q)$, which is the corresponding maximal parabolic subgroup of $G(q)$. Let $\chi$ be the associated permutation character. By \cite[Lemma 2.4]{LLS2}, we have 
\[ 
\chi = \sum_{\phi \in \widehat{W}}n_{\phi}R_{\phi},
\]
where $\widehat{W}$ is the set of complex irreducible characters of the Weyl group $W=W(G)$. Here the $R_{\phi}$ are almost characters of $G(q)$ and the coefficients are given by $n_{\phi} = \la 1^W_{W_M},\phi \ra$, where $W_M$ is the corresponding parabolic subgroup of $W$. The Green functions of $G(q)$ arise by restricting the $R_{\phi}$ to unipotent elements. For the elements $g \in G(q)$ of order $p$ that we are interested in (see Table \ref{tab:class}), L\"{u}beck \cite{Lubeck} has implemented an algorithm of Lusztig \cite{Lus} to compute the relevant Green functions. In particular, we can calculate $\chi(g)$, which is a monic polynomial in $q$ of degree $\dim X(g)$.
We refer the reader to \cite[Section 2]{LLS2} for further details. 

In all cases, one checks that $\a(G,M,g) < f(r)$. For example, if $G=E_8$ and $r=p=2$, then $\dim g^G = 128$ and $g \in G$ is an involution in the $G$-class labelled $A_1^4$. Then for $X = G/P_i$ we get
\[
\begin{array}{lcccccccc} \hline
i & 1 & 2 & 3 & 4 & 5 & 6 & 7 & 8 \\ \hline
\dim X & 78 & 92 & 98 & 106 & 104 & 97 & 83 & 57 \\
\dim X(g) & 38 & 44 & 47 & 51 & 50 & 47 & 40 & 28 \\ \hline
\end{array}
\]
and thus $\a(G,M,g) \leqs \frac{28}{57}$. 

To complete the proof, we may assume $r=p$ and $G = D_4$ or $B_2$. If $G=B_2$ then $p=2$, $\dim D = 4$ and \eqref{e:unip} yields $\dim X(g) \leqs 1$, which is sufficient since $\dim X = 3$. Now suppose $G = D_4$, so we may assume $M=P_1$ or $P_2$, where
$P_1 = U_6A_3T_1$ and $P_2 = U_9A_1^3T_1$. If $r \geqs 5$ then $\dim D \leqs 6$ and the result follows since \eqref{e:unip} gives $\dim X(g) \leqs 1$.

Next assume $r=p=3$, so $\dim g^G = 18$ and $\dim X(g) \leqs 3$ by \eqref{e:unip}. This is good enough if $M=P_2$, but further work is needed when $M=P_1$. As before, we may identify $X=G/P_1$ with the variety of totally singular $1$-spaces in the natural module $V$ and we note that $g$ fixes $U \in X$ if and only if $U$ is contained in the $1$-eigenspace of $g$ on $V$. In terms of the standard basis $\mathcal{B}$ (see \eqref{e:basis}), we may assume that $g = [J_3^2,J_1^2]$ has $1$-eigenspace $\la e_1,f_1, e_2,f_4\ra$ and one checks that $ae_1 +bf_1+ce_2+df_4$ is singular if and only if $ab=0$. This implies that $\dim X(g)=2$ and the result follows.

Finally, suppose $r=p=2$. Here $\dim g^G = 16$ and thus $\dim X(g) \leqs 4$. As in the previous case, this is sufficient for $M=P_2$, but not for $M=P_1$. So let us assume $M=P_1$ and identify $X$ with the variety of totally singular $1$-spaces. Here $g$ is a $c_4$-type involution in the notation of \cite{AS} and we may assume that the $1$-eigenspace of $g$ is spanned by the vectors $e_i+f_i$ for $i=1, \ldots, 4$. An easy calculation shows that the vector $\sum_{i}a_i(e_i+f_i)$ is singular if and only if $\sum_{i}a_i=0$ and we conclude that $\dim X(g)=2$.
\end{proof}

\begin{prop}\label{p:par2}
Let $r \geqs 5$ be a prime divisor of $|G(q)|$ and let $g \in G(q)$ be an element of order $r$ modulo $Z(G)$ with $\dim g^G \geqs \gamma(G,r)$. Then $\a(G,M,g) <\frac{2}{5}$ for every positive dimensional non-parabolic maximal subgroup $M$ of $G$.
\end{prop}

\begin{proof}
Let $M$ be a positive dimensional non-parabolic maximal subgroup of $G$ and set $X=G/M$. Recall that if $G$ is an exceptional group then the possibilities for $M^0$ are listed in Tables \ref{tab:mr} and \ref{tab:nmr}, together with the special cases arising in parts (iii) and (iv) of Theorem \ref{t:ls}. Let $t$ be the rank of $M^0$. 

First assume $G = E_8$, so $\gamma(G,r) \geqs 200$. The trivial bound $\dim (g^G \cap M) \leqs \dim M$ implies that $\dim X(g) < \frac{2}{5}\dim X$ if $\dim M < 128$, so we may assume $M = A_1E_7$ and one checks that the obvious bound $\dim (g^G \cap M) \leqs \dim M -8$ is sufficient.

Next suppose $G = E_7$. By arguing as in the previous case, we may assume $\dim M \geqs 51$, so $M^0$ is one of $E_6T_1$, $A_1D_6$, $A_7$ or $A_1F_4$. If $M^0=A_7$ or $A_1F_4$, then the bound $\dim(g^G \cap M) \leqs \dim M-t$ is sufficient. In the remaining two cases, if $r=p$ then $\gamma(G,r) \geqs 106$ and the previous bound is good enough. For $r \ne p$ we can appeal to  \cite{LLS}. For example, if $M^0 = E_6T_1$ then $\dim X = 54$ and \cite[Theorem 2(II)(b)]{LLS} implies that $\dim X(g) \leqs 20$.

Now assume $G = E_6$. Here we quickly reduce to the case $M = F_4$ with $\dim X = 26$. The $G$-class of each unipotent class in $M$ is recorded in \cite[Table A]{Lawunip} and it is easy to check that $\dim X(g) \leqs 4$ when $r=p$. Now assume $r \neq p$ and set $D = C_G(g)$. If $\dim g^G \geqs 64$ then the trivial bound $\dim(g^G \cap M) \leqs 48$ yields $\dim X(g) \leqs 10$ and the result follows. Therefore, we may assume that $D$ is one of the following:
\[
\begin{array}{lccccc} \hline
D & A_2^2A_1T_1 & A_3A_1T_2 & A_2^2T_2 & A_3T_3 & A_2A_1^2T_2 \\  
\dim g^G & 58 & 58 & 60 & 60 & 62 \\ \hline
\end{array}
\]
Write $M = C_G(\tau)$, where $\tau$ is an involutory graph automorphism of $G$. Without loss of generality, we may assume that $g \in M$ and $\dim (g^G \cap M) = \dim g^M$, so   
\[
\dim X(g) = \dim D - \dim C_D(\tau).
\]
As explained in the proof of \cite[Lemma 5.4]{LLS2}, if $D$ has an $A_3$ factor then $D = A_3T_3$ is the only possibility and we have $C_D(\tau) = C_2T_2$ and $\dim X(g) = 6$. Similarly, if $D = A_2^2A_1T_1$ then $C_D(\tau) = A_2A_1T_1$ and $\dim X(g) = 8$. For $D = A_2A_1^2T_2$ we have $C_D(\tau) = A_2A_1T_1$ or $A_1^2T_2$, which yields $\dim X(g) \leqs 8$. Finally, if $D = A_2^2T_2$ then $C_D(\tau) = A_2T_2$ and once again we deduce that $\dim X(g) = 8$. 

The case $G=F_4$ is very similar. Here $\gamma(G,r) \geqs 40$ and by considering the trivial bound $\dim(g^G \cap M) \leqs \dim M$ we reduce to the cases $M^0 = B_4$, $C_4$ ($p=2$), $D_4$, $\tilde{D}_4$ ($p=2$) and $A_1C_3$. In the latter three cases, the bound $\dim(g^G \cap M) \leqs \dim M-4$ is sufficient. Finally, suppose $M = B_4$ or $C_4$, so $\dim X = 16$. If $r \ne p$ then \cite[Theorem 2(II)(b)]{LLS} gives $\dim X(g) \leqs 4$ and the result follows. On the other hand, if $r=p$ then we must have $M = B_4$ (since $r \geqs 5$); in \cite[Section 4.4]{Law09}, Lawther determines the $G$-class of each unipotent element in $M$ and we deduce that $\dim X(g) \leqs 4$. 

If $G=G_2$ then the bound $\dim(g^G \cap M) \leqs \dim M - t$ is always sufficient, so to complete the proof of the proposition, we may assume $G = D_4$ or $B_2$. If $G=B_2$ then $p=2$ and $M^0 = A_1^2$, so $\dim X = 4$. Moreover, since $\dim g^G =8$ and $\dim(g^G \cap M) \leqs 4$, we deduce that $\dim X(g) = 0$. 

Finally, suppose $G = D_4$. Here $\dim g^G  = 24-2\delta_{5,r}$ and the possibilities for $M^0$ are as follows (up to isomorphism):
\begin{equation}\label{e:d4}
A_3T_1, \, A_1^4, \, T_4, \, B_3 \, (p \ne 2), \, C_3 \, (p=2), \, A_1B_2 \, (p \ne 2), \,  A_1C_2, \, A_2\, (p \ne 3).
\end{equation}
This list is obtained by applying Aschbacher's theorem \cite{As} on maximal subgroups of classical groups (see \cite{LSe} for a shorter proof for algebraic groups). In particular, a maximal subgroup of positive dimension either preserves a geometric structure on the natural module (such as a subspace, or a direct sum decomposition) or the connected component is a simple algebraic group acting irreducibly and tensor indecomposably on the natural module. The list of $8$-dimensional irreducible representations of simple algebraic groups can be read off from \cite{Lu} and it is a simple matter to determine which of these representations preserve a quadratic form.

One checks that the trivial bound $\dim(g^G \cap M) \leqs \dim M$ is sufficient when $\dim M < 13$. Similarly, if $M^0 = A_3T_1$, $A_1B_2$ or $A_1C_2$ then the bound $\dim(g^G \cap M) \leqs \dim M- t$ is good enough. Finally, suppose $M = B_3$ or $C_3$, so $\dim X = 7$. If $r \geqs 7$ then $\dim g^G = 24$ and the bound $\dim(g^G \cap M) \leqs \dim M- 3$ is sufficient. For $r=5$ we have $\dim g^G = 22$ and we note that $\dim (g^G \cap M) \leqs \dim M_{[5]} = 16$, which yields $\dim X(g) \leqs 1$.
\end{proof}

\begin{prop}\label{p:par3}
Let $r \in \{2,3\}$ be a divisor of $|G(q)|$, let $g \in G(q)$ be an element of order $r$ modulo $Z(G)$ with $\dim g^G = \gamma(G,r)$ and let $M$ be a positive dimensional non-parabolic maximal subgroup of $G$. 
\begin{itemize}\addtolength{\itemsep}{0.2\baselineskip}
\item[{\rm (i)}] If $G=D_4$ and $M=B_3$ or $C_3$, then $\a(G,M,g) = \frac{3}{7}$.
\item[{\rm (ii)}] In every other case, either $\a(G,M,g) < f(r)$, or $G=D_4$, $M=A_2$, $r=3$ and $\a(G,M,g) = \frac{2}{5}$.
\end{itemize}
\end{prop}

\begin{proof}
To begin with, let us assume $G$ is one of $E_8, E_7, E_6, F_4$ or $G_2$. We will handle the remaining cases $D_4$ and $B_2$ at the end of the proof. Let $M$ be a positive dimensional non-parabolic maximal subgroup of $G$ and set $X=G/M$.

First assume $G=E_8$. If $M^0 = T_8$ then the trivial bound $\dim (g^G \cap M) \leqs \dim M$ yields
\[
\a(G,M,g) \leqs \frac{\dim G - \gamma(G,r)}{\dim G - 8} = \frac{1}{r}
\]
and the result follows. In each of the remaining cases, we observe that $\dim(g^G \cap M) \leqs \dim M - t$, where $t$ is the rank of $M^0$. In particular, $\a(G,M,r) < f(r)$ if
\begin{equation}\label{e:3m}
\left\{\begin{array}{ll}
3\dim M - 5t < 144 & \mbox{for $r=2$} \\
2\dim M - 5t < 96 & \mbox{for $r=3$.}
\end{array}\right.
\end{equation}

Suppose $M$ has maximal rank, so the possibilities for $M$ are recorded in Table \ref{tab:mr}. In view of \eqref{e:3m}, we may assume that $\dim M \geqs 62$, which implies that $M^0$ is one of $D_8$, $A_1E_7$, $A_8$ or $A_2E_6$. In each case, if $g$ is semisimple then the desired bound follows from \cite[Theorem 2(II)]{LLS}. For example, suppose $M = A_1E_7$, so $\dim X = 112$. If $r=3$ then $C_G(g) = A_8$ and \cite[Theorem 2(II)]{LLS} gives $\dim X(g) \leqs 42$. Similarly, if $r=2$ then $C_G(g) = D_8$ and we see that $\dim X(g) \leqs 56$. 

Now assume $g$ is unipotent. If $r=3$ then $g^G \cap M \subseteq M^0$ since $|M/M^0| \leqs 2$ and it is easy to compute a sufficient upper bound on $\dim(g^G \cap M)$ by considering the conjugacy classes of elements of order $3$ in $M^0$. For example, if $M^0 = D_8$ then $\dim (g^G \cap M) \leqs \dim (D_8)_{[3]} = 80$ and thus $\dim X(g) \leqs 40$. Alternatively, we can compute $\dim X(g)$ precisely by inspecting \cite{Law09}, which gives the $G$-class of each unipotent element in $M^0$. In all cases, it is routine to verify the desired bound. Now assume $r=p=2$, so $g$ is in the class $A_1^4$. As before, we can compute $\dim (g^G \cap M^0)$ via \cite{Law09}, so it just remains to consider $\dim (g^G \cap (M \setminus M^0))$ for $M^0 = A_8$ and $A_2E_6$. In the latter case, we have $\dim (g^G \cap M^0) = 44$ and $\dim (g^G \cap (M \setminus M^0)) = 47$ (see the proof of Lemma \ref{l:mre8}), so $\dim X(g) = 81$ and $\a(G,M,g) = \frac{1}{2}$. Similarly, if $M^0=A_8$ then the proof of \cite[Proposition 5.11]{BGS} gives $\dim (g^G \cap (M \setminus M^0)) = 44$ and once again we conclude that $\a(G,M,g) = \frac{1}{2}$.

To complete the argument for $G=E_8$, we may assume $M^0$ has rank $t<8$ and we note that the possibilities for $M$ are listed in Table \ref{tab:nmr} (together with the special case recorded in part (iv) of Theorem \ref{t:ls}). By considering \eqref{e:3m}, we may assume $M = G_2F_4$. Here one checks that the bound
\[
\dim(g^G \cap M) \leqs \dim (G_2)_{[r]} + \dim (F_4)_{[r]} = 36+10\delta_{3,r}
\]
is sufficient. 

Very similar reasoning applies in all of the remaining cases and no special difficulties arise. Therefore, for brevity we will just give details when $(G,M^0,r)$ is one of the following: 
\[
(E_7,A_7,2),\; (E_7,E_6T_1,2),\; (F_4,D_4,3).
\]

Suppose $(G,M^0,r) = (E_7,A_7,2)$, so $\dim X = 70$ and $M/M^0 = Z_2$. If $p \ne 2$ then \cite[Theorem 2(II)]{LLS} gives $\dim X(g) \leqs 35$, so let us assume $p=2$. Now $\dim (g^G \cap M^0) \leqs \dim (A_7)_{[2]} = 32$ and the proof of \cite[Lemma 3.18]{BGS} gives $\dim (g^G \cap (M \setminus M^0)) = 35$. Therefore, $\a(G,M,g) = \frac{1}{2}$ and the result follows. The case 
$(G,M^0,r) = (E_7,E_6T_1,2)$ is similar. Here $\dim X = 54$ and by applying \cite[Theorem 2(II)]{LLS} we reduce to the case $p=2$. Now $\dim (g^G \cap M^0) \leqs \dim (E_6)_{[2]} = 40$ and the proof of \cite[Lemma 4.1]{LLS} gives $\dim (g^G \cap (M \setminus M^0)) \leqs 43$, whence $\a(G,M,g) \leqs \frac{1}{2}$. Finally, let us assume $(G,M^0,r) =(F_4,D_4,3)$. Here $M = M^0.S_3$ and $\dim X = 24$. The largest conjugacy class in $M$ of elements of order $3$ has dimension $20$ (this is a class of  graph automorphisms), so $\dim (g^G \cap M) \leqs 20$ and we conclude that $\dim X(g) \leqs 8$.

To complete the proof of the proposition, we may assume $G=D_4$ or $B_2$. In the latter case, we have $M^0=A_1^2$ and $r=2$, so $\dim X = 4$, $\dim g^G = 6$ and we note that $\dim (g^G \cap M) \leqs 4$. This gives $\a(G,M,g) \leqs \frac{1}{2}$. 

Now assume $G=D_4$, so $\dim g^G = 16+2\delta_{3,r}$ and the possibilities for $M^0$ are listed in \eqref{e:d4}. If $M^0=T_4$ then $\dim X = 24$ and $\dim (g^G \cap M) \leqs 4$, so $\dim X(g) \leqs 10+2\delta_{2,r}$ and the desired bound follows when $r=2$. Now assume $r=3$ and write $M^0=\prod_{i}M_i$, where each factor $M_i$ is a $1$-dimensional torus. 
Then we may assume $g$ acts nontrivially on the set of factors $\{M_1, M_2, M_3, M_4\}$, in which case $C_{M^0}(g) = T_2$ and $\dim (g^G \cap M) = 2$. This gives $\dim X(g)  = 8$. Similarly, if $M^0=A_1^4$ then $\dim X = 16$ and the bound $\dim(g^G \cap M) \leqs 8$ is sufficient. For $M^0=A_1B_2$ or $A_1C_2$ we observe that $\dim(g^G \cap M) \leqs 2+6=8$ and the result follows since $\dim X(g) \leqs 5+2\delta_{2,r}$.

Now assume $M^0=A_2$, so $M=M^0$, $p \ne 3$ and $\dim X = 20$. If $r=2$ then $\dim (g^G \cap M) = \dim M_{[2]} = 4$ and thus $\dim X(g) = 8$. Now assume $r=3$. Since the embedding of $M$ in $G$ arises from the adjoint representation of $M$, we deduce that $C_G(g) = A_2T_2$. Moreover, $\dim(g^G \cap M) = 6$ and $\dim X(g) = 8$, so $\a(G,M,g) = \frac{2}{5}$ and this case is recorded in part (ii) of the proposition. If $M^0=A_3T_1$ then $\dim X = 12$ and $M/M^0=Z_2$. For $r=3$ we have $\dim(g^G \cap M) \leqs \dim (A_3)_{[3]} = 10$ and we get $\dim(g^G \cap M^0) \leqs 8$ and $\dim(g^G \cap (M \setminus M^0)) \leqs 9$ if $r=2$. These bounds are sufficient. 

Next suppose $M=M^0=B_3$ and $p \ne 2$, so $\dim X = 7$. First assume $r=3$. If $p=3$ then $g$ has Jordan form $[J_3^2,J_1^2]$ on the natural module for $G$ and we see that $g^G \cap M$ is the set of elements in $M$ with Jordan form $[J_3^2,J_1]$ on the natural module for $B_3$ (we may assume $M$ is the stabilizer of a nondegenerate $1$-space). Therefore, $\dim(g^G \cap M) = 14$ and we conclude that $\a(G,M,g) = \frac{3}{7}$. Similarly, one checks that $\a(G,M,g) = \frac{3}{7}$ when $p \ne 3$ and $C_G(g) = A_1^3T_1$. In particular, this special case is highlighted in part (i) of the proposition. Now assume $r=2$. Here each element in $g^G \cap M$ has Jordan form $[-I_4,I_3]$ on the natural module for $M$, so $\dim (g^G \cap M) = 12$ and once again we conclude that $\a(G,M,g) = \frac{3}{7}$. The case $M=C_3$ with $p=2$ is very similar and we get $\a(G,M,g) = \frac{3}{7}$ for $r=2,3$ (with $C_G(g) = A_1^3T_1$ when $r=3$ and $p \ne 3$). 
\end{proof}

\vs

This completes the proof of Theorem \ref{t:main6}, and Corollary \ref{cor:1} follows immediately.

\vs

\subsection{Proof of Theorem \ref{t:main7}}\label{ss:thm7}

Finally, we will use Corollary \ref{cor:1} to prove Theorem \ref{t:main7}, which is our main result on the random generation of finite simple exceptional groups of Lie type. To do this, we adopt the approach introduced in \cite{GLLS}.  

To begin with, we will exclude the Suzuki and Ree   groups. Let $G(q)$ be a finite quasisimple exceptional group of Lie type over $\mathbb{F}_q$, where $q=p^f$ for some $f \geqs 1$. Let $G$ be the ambient simply connected simple algebraic group of exceptional type over $k$, the algebraic closure of $\mathbb{F}_p$, and let $\s$ be a Steinberg endomorphism of $G$ with $G_{\s} = G(q)$. Let $r$ and $s$ be prime divisors of $|G(q)/Z(G(q))|$ and assume $(r,s) \ne (2,2)$. 

First assume that $p$ does not divide $rs$. Fix conjugacy classes $C_r$ and $C_s$ of $G$ such that the following conditions are satisfied:
\begin{itemize}\addtolength{\itemsep}{0.2\baselineskip}
\item[(a)] $C_r$ and $C_s$ contain elements of order $r$ and $s$ modulo $Z(G)$, respectively;
\item[(b)] $C_r(q):=C_r \cap G(q)$ and $C_s(q):=C_s \cap G(q)$ are nonempty;
\item[(c)] $\dim C_r = \mathcal{C}(G,r,q)$ and $\dim C_s = \mathcal{C}(G,s,q)$.
\end{itemize}
Note that $\dim C_r \geqs \gamma(G,r)$ and $\dim C_s \geqs \gamma(G,s)$ (see Theorem \ref{t:main6}(i)). By combining Corollary \ref{cor:1} with Theorem \ref{t:main3}, it follows that there is a nonempty open subset $U$ of $C_r \times C_s$ such that for all $x \in U$, $G(x)$ is not contained in a positive dimensional proper closed subgroup of $G$. Therefore, for any algebraically closed
field $k'$ properly containing $k$, there is a dense set of elements $x \in C_r(k') \times C_s(k')$ with $G(x)=G(k')$. By combining Theorem \ref{t:algebraic} with  \cite[Theorems 1 and 2]{GLLS}, we deduce that the proportion of pairs in $C_r(q) \times C_s(q)$ which generate $G(q)$ tends to $1$ as $q$ tends to infinity (recall we are only considering
 $q$ for which $C_r(q) \times C_s(q)$ is nonempty).   

Now assume $C_r'$ and $C_s'$ is any other pair of conjugacy classes of elements in $G$ of orders $r$ and $s$ (modulo $Z(G)$) such that $C_r'(q)$ and $C_s'(q)$ are nonempty. Then 
\[
\dim C_r' + \dim C_s' < \mathcal{C}(G,r,q) + \mathcal{C}(G,s,q)
\]
and thus Lang-Weil \cite{LW} implies that $|C_r'(q) \times C_s'(q)|$ is much smaller than the total number of pairs of elements of the appropriate orders (in particular, this ratio goes to $0$ as $q \rightarrow \infty$). Therefore, the proportion of pairs of elements of orders $r$ and $s$ (modulo $Z(G)$) which generate $G(q)$ (equivalently, generate $G(q)/Z(G(q))$) tends to $1$ as $q$ increases (again, under the assumption that $r$ and $s$ divide $|G(q)/Z(G(q))|$).

Similarly, if $p$ divides $rs$ then the same argument applies, but the characteristic of the underlying field $\mathbb{F}_q$ is now fixed, which means that we only need to apply \cite[Theorem 1]{GLLS}. 

Finally, suppose $G(q)/Z(G(q))$ is a Suzuki or Ree group. Here $p$ is fixed and the same argument applies, using \cite[Theorem 1]{GLLS} to show that for the largest conjugacy classes $C_r$ and $C_s$, the proportion of pairs which generate
goes to $1$ as $q$ increases (under the assumption that $r$ and $s$ both divide $|G(q)/Z(G(q))|$). There is a version of Lang-Weil which applies
in this case as well but one can just check directly that for any pair of conjugacy classes $C_{r}'$ and $C_{s}'$ consisting
of elements of orders $r$ and $s$ with  $\dim C_r' + \dim C_s' < \mathcal{C}(G,r,q) + \mathcal{C}(G,s,q)$ we have
$|C_{r}'(q)||C_{s}'(q)| \ll |C_{r}(q)||C_{s}(q)|$ and the result follows. 

\vs

This completes the proof of Theorem \ref{t:main7}.

\end{document}